\documentclass[11pt]{amsart}
\usepackage{amssymb,mathrsfs,graphicx,enumerate}
\usepackage{amsmath,amsfonts,amssymb,amscd,amsthm,bbm}
\usepackage[retainorgcmds]{IEEEtrantools}
\usepackage{colortbl}
\usepackage{caption}
\usepackage{subcaption}
\usepackage{tikz}
\allowdisplaybreaks

\title[Kuramoto models via mean field dynamics]{Mean field Kuramoto models on graphs}

\author[Li]{Wuchen Li}
\address[Wuchen Li]{\newline Department of Mathematics \newline University of South Carolina, Columbia, SC 29208, USA}
\email{wuchen@mailbox.sc.edu}

\author[Park]{Hansol Park}
\address[Hansol Park]{\newline Department of Mathematics \newline Simon Fraser University, 8888 University Dr, Burnaby, BC V5A 1S6, Canada}
\email{hansol\_park@sfu.ca}
\email{hansol960612@snu.ac.kr}

\newtheorem{theorem}{Theorem}[section]
\newtheorem{lemma}{Lemma}[section]

\newtheorem{proposition}{Proposition}[section]
\newtheorem{remark}{Remark}[section]
\newtheorem{example}{Example}[section]

\newcommand{\bbr}{\mathbb R}

\newcommand{\bbs}{\mathbb S}

\makeatletter
\@namedef{subjclassname@2020}{\textup{2020} Mathematics Subject Classification}
\makeatother

\begin{document}

\date{\today}


\thanks{\textbf{Acknowledgment.} W. Li is supported by a start up funding in University of South Carolina and NSF RTG: 2038080. And H. Park is supported by Pacific Institute for the Mathematical Science(PIMS), Canada postdoctoral fellowship.}

\begin{abstract}
One of a classical synchronization model is the Kuramoto model. We propose both first and second order Kuramoto dynamical models on graphs using discrete optimal transport dynamics. We analyze the synchronization behaviors for some examples of Kuramoto models on graphs. We also provide a generalized Hopf-Cole transformation for discrete optimal transport systems. Focus on the two points graph, we derive analytical formulas of the Kuramoto dynamics with various potential induced from entropy functionals. Several numerical examples for the Kuramoto model on general graphs are presented.

\end{abstract}

\maketitle \centerline{\date}


\section{Introduction}\label{sec:1}
\setcounter{equation}{0}
Synchronization phenomena of interacting particles plays essential roles in physics, biology, social systems, and artificial intelligence (AI). This field of research have been studied intensively by two pioneers: A. Winfree \cite{Winf} and Y. Kuramoto \cite{Ku1, Ku2}. One of the famous synchronization model is the Kuramoto model. This model is a synchronization model on the unit circle given as follows:
\begin{align}\label{A-1}
\begin{cases}
\displaystyle\dot{\varphi}_i=\nu_i+\frac{\kappa}{N}\sum_{k=1}^N\sin(\varphi_k-\varphi_i)\quad\forall~t>0,\\
\varphi_i(0)=\varphi_i^0\in\bbr\quad \forall~i=1, 2, \dots, N.
\end{cases}
\end{align}
The emergent behavior of system \eqref{A-1} has been studied in \cite{H-H-K, H-L-X}. For the identical case (i.e. $\nu_i=0$ for all $i$), we can prove that the difference between phases $\varphi_i-\varphi_j$ converges to $2\pi m$ exponentially for some $m\in \mathbb{Z}$ under suitable initial conditions. If the number of particles $N$ tends to infinity, we have the following mean field dynamical model of the identical case of \eqref{A-1}:
\begin{align}\label{A-2}
\begin{cases}
\partial_t \rho_t(\varphi)+\partial_\varphi(\rho_t(\varphi)L[\rho_t](\varphi))=0,\\
\displaystyle L[\rho_t](\varphi)=-\kappa\int_{\mathbb{S}^1}\sin(\varphi-\varphi_*)\rho_t(\varphi_*)d\varphi_*,\quad\varphi\in\bbr,
\end{cases}
\end{align}
where $\rho_0$ is the initial distribution. Similarly, the emergent behavior of equation \eqref{A-2} is given as 
\begin{align}\label{A-3}
\rho_t\to\delta_{\varphi^\infty}
\end{align}
in \cite{H-K-M-P}, where $\delta_{\varphi^\infty}$ is a Dirac delta function concentrated at a constant point $\varphi^\infty$. We can interpret it as the probability density function defined on the unit circle is concentrating into one point $\varphi^\infty$ as time goes to infinity. In practice, a large population of particles are often settled on a discrete spatial domain; e.g. a simple finite graph. A natural question arise: 

{\em (Q) What is the synchronization phenomena of a large population of interacting particles on a discrete domain, such as finite graphs?  }

This paper provides a proper form of a discretized system \eqref{A-2} on finite graphs using optimal transport dynamics \cite{Vi, chow2012, M}.  We show that if we discretize system \eqref{A-2} onto a discrete graph with $n$ vertices, then the corresponding emergent behavior to \eqref{A-3} can be expressed as follows (up to permutation):
\begin{align*}\label{A-4}
(\rho_1(t), \rho_2(t)\cdots, \rho_n(t))\to(1, 0, \cdots, 0)\quad\text{as}\quad t\to\infty.
\end{align*}
To study discrete dynamical system \eqref{A-2}, we use the tools of optimal transport dynamics on graphs; see
\cite{L-L, L-M, Mass, Li-entropy}. In Section \ref{sec:2.2}, we express system \eqref{A-2} as both a gradient flow system and a Hamiltonian system. From this fact, we have the following first order model on the discrete model as follows:
\[
\begin{cases}
\displaystyle\frac{d\rho_j}{dt}=\kappa\sum_{k=1}^n\theta_{jk}(\rho_j-\rho_k),\quad t>0,\\
\rho_j(0)=\rho_j^0, \quad j=1, 2, \cdots, N.
\end{cases}
\]
Also, we can follow the discretization introduced in \cite{C-L-Z, CLMZ} to obtain the discretized system of \eqref{A-2} below:
\begin{align*}\label{A-5}
\begin{cases}
\displaystyle \frac{d\rho_j}{dt}-\sum_{l\in N_j}\omega_{jl}(S_j-S_l)\theta_{jl}=0,\vspace{0.2cm}\\
\displaystyle \frac{dS_j}{dt}+\frac{1}{2}\sum_{l\in N_j}\omega_{jl}(S_j-S_l)^2\frac{\partial \theta_{jl}}{\partial \rho_j}=\frac{\partial}{\partial\rho_j}\left(\frac{1}{2}\sum_{(k, l)\in E}\frac{\omega_{kl}}{2}(\partial_k\mathcal{F}-\partial_l\mathcal{F})^2\theta_{kl}\right),\vspace{0.2cm}\\
\displaystyle\rho_j(0)=\rho_j^0,\quad S_j(0)=S_j^0,\quad S_j(0)=-\frac{\partial\mathcal{F}(\rho)}{\partial\rho_j}\Bigg|_{\rho=\rho(0)},
\end{cases}
\end{align*}
where $\mathcal{F}(\rho)$ is a potential functional defined on a graph, $\rho_j$ is the probability density on node $j$, and $S_j$ is the potential function on node $j$. Also, $\theta_{ij}$ represents the probability weight of the edge connecting $i^{th}$ and $j^{th}$ vertices. For example, $\theta_{ij}=\frac{1}{2}(\rho_i+\rho_j)$ can be picked, however, in this paper we consider the general probability weight given as $\theta_{ij}=\theta(\rho_i, \rho_j)$ for some function $\theta$. We study this discretized model in this paper intensively.\\

The main results of this paper are three-fold. First, we develop first and second order Kuramoto model on graphs following discrete optimal transport dynamics. We also provide emergent behaviors (synchronization behaviors) of the Kuramoto model on graphs. Second, we provide the Hopf-Cole transformation for discrete optimal transport system. Third, we derive analytical formulas of the Kuramoto dynamics on the two points graph with various entropy functionals.\\

This paper is organized as follows. In Section \ref{sec:2}, we review some aggregation models, optimal transport dynamics on discrete graph, and the generalized Hopf-Cole transformation. We construct first and second order Kuramoto model on finite graphs and study their emergent behaviors in Section \ref{sec:3}. We also apply the generalized Hopf-Cole transformation on discrete graph in Subsection \ref{sec:3.4}.  In Section \ref{sec:4}, we provide analytical formulas of the Kuramoto dynamics on the two points graph. Several numerical examples are presented in Section \ref{sec:5}. 

\section{Review}\label{sec:2}
\setcounter{equation}{0}

In this section, we provide some reviews on concepts and results on aggregation models and optimal transport equations.

For the intuition, consider a case of classical mechanics on $\mathbb{R}^d$.
We first consider a gradient flow system
\[
\begin{cases}
\dot{x}=-\nabla U(x),\\
x(0)=x^0\in\bbr^d.
\end{cases}
\]
We second study a Hamiltonian system. Suppose the Hamiltonian is 
\[
H(x, p)=\frac{1}{2}p^2-\frac{1}{2}|\nabla U(x)|^2.
\]
The Hamiltonian system satisfies
\begin{align}\label{D-6}
\begin{cases}
\displaystyle\dot{x}=p,\quad \dot{p}=-\frac{1}{2}\nabla|\nabla U(x)|^2,\\
x(0)=x^0,\quad p(0)=p^0.
\end{cases}
\end{align}
If we impose the following assumption:
\[
p^0=-\nabla U(x^0),
\]
then this Hamiltonian system can be reduced to $\dot{x}=-\nabla U(x)$. See Lemma \ref{LA.1} for detail proof. In this sense, any gradient flow system can be expressed as a particular case of Hamiltonian system. 
\subsection{Review on aggregation models}\label{sec:2.1}
In this subsection, we review some gradient flow type aggregation models. The Kuramoto model is the simplest synchronization model defined on $\bbs^1$. Consider 
\begin{align}\label{B-1}
\begin{cases}
\displaystyle\dot{\varphi}_i=\nu_i+\frac{\kappa}{N}\sum_{k=1}^N\sin(\varphi_k-\varphi_i),\\
\varphi_i(0)=\varphi_i^0,\quad \forall~ i\in\mathcal{N}:=\{1, 2, \cdots, N\},
\end{cases}
\end{align}
where $N$ is the number of particles, $\nu_i$ is the natural frequency of the $i^{th}$ particle, $\kappa$ is the coupling strength, and $\varphi_i$ is the phase of the $i^{th}$ particle. If the natural frequency is identical, i.e., $\nu_i\equiv\nu$ for all $i\in \mathcal{N}$, then solution of \eqref{B-1} $\Phi:=\{\varphi_i\}_{i=1}^N$ exhibits the complete synchronization:
\[
\lim_{t\to\infty}(e^{\mathrm{i}\varphi_i}-e^{\mathrm{i}\varphi_j})=0.
\]
Also, if the coupling strength $\kappa$ is large enough, then 
\[
\lim_{t\to\infty}(\dot{\varphi}_i-\dot{\varphi}_j)=0.
\]
The Kuramoto model can be expressed as a gradient flow system:
\begin{align*}
\begin{cases}
\dot{\Phi}=-\nabla_{\Phi}\mathcal{V},\quad t>0,\\
\displaystyle\mathcal{V}(\Phi)=-\sum_{i=1}^N\nu_i\varphi_i-\frac{\kappa}{2N}\sum_{i, j=1}^N\cos(\varphi_i-\varphi_j).
\end{cases}
\end{align*}

When $N$ goes to infinity, we can obtain the mean field version of system \eqref{B-1}. The mean field version of the identical Kuramoto model($\nu_i\equiv0$) is the Kuramoto-Sakaguchi model \cite{Sa} given as follows:
\begin{align}\label{B-2}
\begin{cases}
\partial_t \rho_t(\varphi)+\partial_\theta(\rho_t(\varphi)L[\rho_t](\varphi))=0,\\
\displaystyle L[\rho_t](\varphi)=-\kappa\int_{\mathbb{S}^1}\sin(\varphi-\varphi_*)\rho_t(\varphi_*)d\varphi_*\quad\forall~\varphi\in\bbr,
\end{cases}
\end{align}
where $\rho_0$ is an initial probability distribution function defined on $\mathbb{S}^1$.\\

A generalization of system \eqref{B-1} to general sphere $\bbs^{d-1}$ with $d\geq2$ is introduced in \cite{Lo1}:
\[
\begin{cases}
\displaystyle\dot{x}_i=\Omega_i x_i+\frac{\kappa}{N}\sum_{k=1}^N\left(x_k-\langle x_i, x_k\rangle x_i\right),\\
x_i(0)=x_i(0)\in\bbs^{d-1}\subset \bbr^d,\quad\forall~i\in\mathcal{N},
\end{cases}
\]
where $\Omega_i$ is a skew-symmetric matrix of size $d\times d$. We call this system {\it the swarm sphere model}. The mean field version of the identical swarm sphere model ($\Omega_i\equiv0$) is given below:
\begin{align}\label{B-3}
\begin{cases}
\partial_t\rho_t(x)+\nabla_{\bbs^d}\cdot(\rho_t(x)L[\rho_t](x))=0,\\
\displaystyle L[\rho](x)=\kappa\int_{\bbs^d} (y-\langle x, y\rangle x)\rho(y)dy\quad\forall~x\in\bbs^d,
\end{cases}
\end{align}
where $\rho_0$ is the initial distribution on $\bbs^d$. We will use this system \eqref{B-3} to construct the consensus model on a regular $n$-simplex.

\subsection{Optimal transport dynamics.} \label{sec:2.2}
In this subsection, we present some previous results on optimal transport dynamics.  We first consider a gradient flow system on Wasserstein space. We then demonstrate that any Wasserstein gradient flows can be expressed as Hamiltonian flows.

Consider the following Wasserstein gradient flow system:  
\begin{align}\label{D-1}
\partial_t\rho_t=-\mathrm{grad}_W\mathcal{F}(\rho_t)=\mathrm{div}\left(\rho_t\nabla\left(\frac{\delta\mathcal{F}(\rho_t)}{\delta\rho_t}\right)\right),
\end{align}
where the initial distribution is $\rho_0$, $\mathcal{F}$ is a given 
functional, $\frac{\delta}{\delta\rho}$ is the $L^2$ first variation operator w.r.t. $\rho$, and $\mathrm{grad}_W$ is a gradient operator in Wasserstein space:
\[
\mathrm{grad}_W\mathcal{F}[\rho]=-\mathrm{div}\left(\rho \nabla\frac{\delta\mathcal{F}[\rho]}{\delta\rho}\right).
\]

Consider the Wasserstein Hamiltonian system
\begin{align}\label{D-1-2}
\begin{cases}
\partial_t\rho_t=\partial_S\mathcal{H}(\rho_t, S_t),\\
\partial_t S_t=-\partial_\rho\mathcal{H}(\rho_t, S_t),
\end{cases}
\end{align}
where the Hamiltonian is defined as
\[
\mathcal{H}(\rho, S)=\frac{1}{2}\int\Big(|\nabla S(x)|^2-|\nabla \delta \mathcal{F}(\rho)|^2\Big)\rho(x)dx.
\]
In other words, 
\begin{align}\label{D-2}
\begin{cases}
\partial_t\rho_t+\mathrm{div}(\rho_t\nabla S_t)=0,\\
\displaystyle\partial_t S_t+\frac{1}{2}|\nabla S_t|^2=\frac{\delta}{\delta\rho_t}\left(\frac{1}{2}\int\left|\nabla \left(\frac{\delta\mathcal{F}(\rho_t)}{\delta\rho_t}\right)\right|^2\rho_t(x)dx\right).
\end{cases}
\end{align}
If the initial data $(\rho_0, S_0)$ satisfies
\begin{align}\label{D-2-1}
S_0=-\frac{\delta\mathcal{F}(\rho_0)}{\delta\rho_0},
\end{align}
then system \eqref{D-2} can be reduced to \eqref{D-1}; see Lemma \ref{LA.2} for the detail proof. So, we conclude that any Wasserstein gradient flows can also be expressed as particular Hamiltonian systems.\\

\subsection{Generalized Hopf-Cole transformation}\label{sec:2.3}
In the remained part of this section, we review generalized Hopf-Cole transformations \cite{C-C-T,L-L}. It is to rewrite the Hamiltonian system in term of entropy functions. 
Define a pair of functions $(\eta, \eta^*)$, such that 
\begin{align*}\label{D-3}
\begin{cases}
\delta\mathcal{F}(\rho)=\delta\mathcal{F}(\eta)+\delta\mathcal{F}(\eta^*),\\
S=\delta\mathcal{F}(\eta)-\delta\mathcal{F}(\eta^*).
\end{cases}
\end{align*}
The Hamiltonian system \eqref{D-1-2} in term of $(\eta, \hat\eta)$ satisfies 
\[
\begin{cases}
\partial_t\eta_t=\sigma(\eta_t, \eta_t^*)\partial_{\eta^*}\mathcal{K}(\eta_t, \eta_t^*),\\
\partial_t\eta_t^*=-\sigma(\eta_t^*, \eta)\partial_{\eta}\mathcal{K}(\eta_t, \eta_t^*),
\end{cases}
\]
where $\mathcal{K}(\eta, \eta^*)=\mathcal{H}(\rho, S)$ and 
\[
\sigma(\eta, \eta^*)(x, w)=-\frac{1}{2}\iint[\delta^2\mathcal{F}(\eta)]^{-1}(x, y)\delta^2\mathcal{F}(\rho)(y, z)[\delta^2\mathcal{F}(\eta^*)]^{-1}(z, w)dydz.
\]
This transformation can be not well-defined when the inverse of $\delta\mathcal{F}$ is not well-defined. One can also set new variables $\xi$ and $\xi^*$, such that
\begin{align}\label{D-3-0}
\xi=\delta\mathcal{F}(\eta),\quad \xi^*=\delta\mathcal{F}(\eta^*).
\end{align}
Then, the Hamiltonian system \eqref{D-2-1} in term of $(\xi,\xi^*)$ satisfies 
\begin{align}\label{D-3-1}
\begin{cases}
\partial_t\xi(x)\displaystyle=\nabla\xi(x)\cdot\nabla\xi^*(x)-\int[\delta^2\mathcal{F}](x, u)\nabla\cdot(\rho(u)\nabla \xi(u))du,\\
\partial_t\xi^*(x)\displaystyle=-\nabla\xi(x)\cdot\nabla\xi^*(x)+\int[\delta^2\mathcal{F}](x, u)\nabla\cdot(\rho(u)\nabla \xi^*(u))du.
\end{cases}
\end{align}
For the detail proof, we refer \cite{C-C-T}. In this paper, we perform the discretized analog of the change of variable \eqref{D-3-1} in Section \ref{sec:3.4}.

\subsection{Aggregation models via optimal transport dynamics}
Actually, many kinetic aggregation models can be expressed as Wasserstein gradient flow system \eqref{D-1}. In detail, the kinetic Kuramoto model \eqref{B-2} satisfies
\begin{align*}
\begin{cases}
\displaystyle\partial_t\rho_t(\theta)=\partial_\theta\left(\rho_t(\theta)\partial_{\theta}\left(\frac{\delta}{\delta\rho_t(\theta)}\mathcal{V}(\rho_t)\right)\right)=-\mathrm{grad}_W(\mathcal{V}(\rho_t)),\\
\displaystyle\mathcal{V}(\rho)=-\frac{\kappa}{2}\iint \cos(\theta_*-\theta_{**})\rho(\theta_*)\rho(\theta_{**})d\theta_*d\theta_{**}.
\end{cases}
\end{align*}
Similarly, the kinetic swarm sphere model \eqref{B-3} forms
\begin{align}\label{D-2-2}
\begin{cases}
\partial_t\rho_t(x)=\nabla_{\bbs^d}\cdot\Big(\rho_t(x)\nabla_{\bbs^d}\frac{\delta}{\delta\rho_t}\mathcal{V}[\rho_t]\Big)=-\mathrm{grad}_W \mathcal{V}[\rho_t](x),\\
\displaystyle\mathcal{V}[\rho]=\frac{\kappa}{2}\iint_{(\bbs^d)^2} \|x-y\|^2\rho(x)\rho(y)dxdy.
\end{cases}
\end{align}

 Recall the kinetic swarm sphere model is given as \eqref{B-3} and its gradient flow formulation can be expressed as \eqref{D-2-2}. Since system \eqref{D-2-2} is a special case of \eqref{D-1}, we can express this system as a Hamiltonian flow formulation using \eqref{D-2} as follows:
\begin{align}\label{F-2-1}
\begin{cases}
\displaystyle \partial_t\rho_t+\mathrm{div}(\rho_t\nabla S_t)=0,\\
\displaystyle \partial_t S_t+\frac{1}{2}|\nabla S_t|^2=\frac{\delta}{\delta\rho_t}\left(\frac{1}{2}\int \left|\nabla\left(\frac{\delta\mathcal{V}(\rho_t)}{\delta\rho_t}\right)\right|^2\rho_t(x)dx\right),
\end{cases}
\end{align}
where $\mathcal{V}$ is defined in \eqref{D-2-2}.

\section{Main results: Kuramoto models on graphs}\label{sec:3}
\setcounter{equation}{0}

In section \ref{sec:3}, we present the main result of this paper.

\subsection{Modeling of synchronization model on discrete graph}\label{sec:3.1}
 We first connect aggregation models with optimal transport dynamics. 

 Let the regular $n$-simplex be given on $\bbs^{n-2}\subset\bbr^{n-1}$($n\geq2$). Let $n$ points $x_1, x_2, \cdots, x_n$ satisfy 
\[
\{x_1, \cdots, x_n\}\subset \bbs^{n-2},\quad \|x_i-x_j\|=L_n\quad\forall~i\neq j
\]
for some positive constant $L_n$. Actually, this constant can be calculated for all $n\geq2$. e.g., $n=3$, it is an equilateral triangle with side length $L_3=\sqrt{3}$. We consider a graph $G_{n}$, which consists $\{x_1, \cdots, x_{n}\}$ as the set of vertices and $\{(x_i, x_j):1\leq i<j\leq n\}$ as the set of edges. Since \eqref{F-2-1} is a special case of \eqref{D-2}, we have the following discretization on $G_n$ of system \eqref{F-2-1}:
\begin{align}\label{F-2-2}
\begin{cases}
\displaystyle\frac{d\rho_j}{dt}-\sum_{\substack{l=1\\l\neq j}}^{n}\omega_{jl}(S_j-S_l)\theta_{jl}=0,\\
\displaystyle\frac{dS_j}{dt}+\frac{1}{2}\sum_{\substack{l=1\\l\neq j}}^{n}(S_j-S_l)^2\frac{\partial\theta_{jl}}{\partial\rho_j}=\frac{\partial}{\partial\rho_j}\left(\frac{1}{4}\sum_{\substack{k, l=1\\k\neq l}}^{n}\omega_{kl}\Big(\partial_k\tilde{\mathcal{F}}_n-\partial_l\tilde{\mathcal{F}}_n\Big)^2\theta_{kl}\right),
\end{cases}
\end{align}
where potential $\tilde{\mathcal{F}}$ is given as
\[
\tilde{\mathcal{F}}_n(\rho)=\frac{\kappa}{2}\sum_{i, j=1}^{n}\|x_i-x_j\|^2\rho_i\rho_j=\frac{\kappa}{2}\sum_{i, j=1}^{n}L_n^2(1-\delta_{ij})\rho_i\rho_j.
\]
Here $\rho_j$ and $S_j$ are the probability densityh and potential function on node $j$, respectively. Since the potential function will be differentiated, we add some constant to define another potential function:
\begin{align}\label{F-3}
\mathcal{F}_n[\rho]=\tilde{\mathcal{F}}_n[\rho]-\frac{\kappa}{2} L_n^2=-\frac{\kappa}{2}\sum_{i, j=1}^{n}\delta_{ij}\rho_i\rho_j=-\frac{\kappa L_n^2}{2}\sum_{i=1}^{n}\rho_i^2.
\end{align}
Also, from $\|x_i-x_j\|\equiv L_n$ for all $i\neq j$, we set $\omega_{ij}=\frac{1}{L_n^2}$ for all $i\neq j$. We rescale the system to set $L_n=1$. For the convenience, we set $\theta_{ii}=\rho_i$. From the definition of $\mathcal{F}_n$, we get
\[
\partial_i\mathcal{F}_n=-\kappa \rho_i\quad\forall~1\leq i\leq n.
\]
From these additional assumptions, we can reduce system \eqref{F-2-2} as follows:
\begin{align}\label{F-4}
\begin{cases}
\displaystyle\frac{d\rho_j}{dt}-\sum_{l=1}^{n}(S_j-S_l)\theta_{jl}=0,\\
\displaystyle\frac{dS_j}{dt}+\frac{1}{2}\sum_{l=1}^{n}(S_j-S_l)^2\frac{\partial\theta_{jl}}{\partial\rho_j}=\frac{\partial}{\partial\rho_j}\left(\frac{\kappa^2}{4}\sum_{k, l=1}^{n}\Big(\rho_k-\rho_l\Big)^2\theta_{kl}\right).
\end{cases}
\end{align}

If the initial data of system \eqref{F-4} satisfy
\begin{align}\label{F-4-1-0-0-0}
S_j(0)=-\partial_j\mathcal{F}_n\Big|_{t=0}=\kappa\rho_j(0)\quad\forall~1\leq j\leq n,
\end{align}
then the system can be reduced to the gradient flow formulation(first order system). 

\subsection{First order dynamics}\label{sec:3.2}

In this subsection, we study the emergent behavior of system \eqref{F-4} with \eqref{F-4-1-0-0-0}. From Lemma \ref{LA.2}, we have
\[
S_j(t)=\kappa \rho_j(t)
\]
for all $1\leq j\leq n$ and $t\geq0$. We substitute this relation into \eqref{F-4}, then we have
\begin{align}\label{F-4-1}
\frac{d\rho_j}{dt}=\kappa\sum_{k=1}^{n}\theta_{jk}(\rho_j-\rho_k).
\end{align}
From a simple calculation, we get
\[
\frac{d}{dt}\sum_{j=1}^n \rho_j^2=\kappa\sum_{j, k=1}^n \theta_{jk}\rho_j(\rho_j-\rho_k).
\]
We assume the following conditions on $\theta_{ij}=\theta(\rho_i, \rho_j)$:

\noindent$\bullet$ Symmetry: $\theta_{ij}=\theta_{ji}$,

\noindent$\bullet$ Lipschitz continuity: $\theta_{ij}$ is Lipschitz continuous function,

\noindent$\bullet$ $\theta_{ij}\geq0$ and the equality only holds either $\rho_i=0$ or $\rho_j=0$.\\

Then, we have
\[
\frac{d}{dt}\sum_{j=1}^n\rho_j^2=\frac{\kappa}{2}\sum_{j, k=1}^n\theta_{jk}(\rho_j-\rho_k)^2\geq0.
\]
On the other hand, we can easily check that if $\theta_{jk}$ is continuous, then the derivative of $\frac{d}{dt}\sum_{j=1}^n\rho_j^2$ is uniformly continuous and we can apply Barbalat's lemma to get
\[
\lim_{t\to\infty}\theta_{jk}(t)(\rho_j(t)-\rho_k(t))^2=0
\]
and this again implies either 
\[
\lim_{t\to\infty}\min(\rho_j(t), \rho_k(t))=0\quad\text{or}\quad \lim_{t\to\infty}(\rho_j(t)-\rho_k(t))=0.
\]
We also use
\[
\sum_{i=1}^{n}\rho_i=1
\]
to get that there exists at least one $1\leq j\leq n$ such that $\rho_j$ does not converges to zero. The set of equilibrium of this system can be expressed as follows:
\[
\mathcal{E}=\bigcup_{m=1}^{n}\mathcal{E}_m,
\]
where $\mathcal{E}_n$ consists $(\rho_1, \cdots, \rho_{n})$ if there exist distinct indices $i_1, i_2, \cdots, i_m$ such that $\rho_{i_1}=\cdots=\rho_{i_m}=\frac{1}{m}$ and other $\rho_j$ are zero. From direct calculations, we have the following lemma.

\begin{lemma}
Let $\rho$ be a solution to \eqref{F-4-1} with initial data $\rho(0)$. If $\rho_i(0)=\rho_j(0)$, then $\rho_i(t)=\rho_j(t)$ for all $t\geq0$.
\end{lemma}

Now, we assume the following additional condition:\\

\noindent$\bullet$ $\theta_{ij}$ is an increasing function of $\min(\rho_i, \rho_j)$, i.e. $\theta(\rho_i, \rho_j)=\tilde{\theta}(\min(\rho_i, \rho_j))$ for some $\tilde{\theta}$ and satisfies
\[
\min(\rho_i, \rho_j)\geq\min(\rho_k, \rho_l)\quad\Longrightarrow\quad\theta_{ij}\geq\theta_{kl}.
\]

As examples, we have $\theta_{ij}=\theta(\rho_i, \rho_j)=\mathrm{min}(\rho_i, \rho_j)$. This $\theta_{ij}$ satisfies all conditions that we assumed. From the additional condition, we have the following property.
\begin{lemma}\label{L3.4}
Suppose that the initial data $\{\rho_j(0)\}_{j=1}^n$ satisfy
\[
\rho_1(0)=\max_{1\leq j\leq n}\rho_j(0)\quad\text{and}\quad \rho_1(0)>\max_{1<j\leq n}\rho_j(0),
\]
and let $\rho$ be a solution to \eqref{F-4-1} with initial data $\rho(0)$. Then
\[
\rho_1(t)-\max_{1<j\leq n}\rho_j(t)\geq\rho_1(0)-\max_{1<j\leq n}\rho_j(0).
\]
\end{lemma}

\begin{proof}
Without loss of generality, we set
\[
\rho_2(t)=\max_{1<j\leq n}\rho_j(t)\quad\text{and}\quad \rho_1(t)>\rho_2(t)\quad \forall~t\in[0, \epsilon)
\]
for a positive $\epsilon$. From simple calculations, we have
\begin{align*}
\dot{\rho}_1&=\kappa\sum_{k=1}^{n}\theta_{1k}(\rho_1-\rho_k)=\kappa\sum_{k=1}^{n}\tilde\theta(\rho_k)(\rho_1-\rho_k)\\
\dot{\rho}_2&=\kappa\sum_{k=1}^{n}\theta_{2k}(\rho_2-\rho_k)=\kappa\tilde{\theta}(\rho_2)(\rho_2-\rho_1)+\kappa\sum_{k=2}^{n}\tilde\theta(\rho_k)(\rho_2-\rho_k).
\end{align*}
This implies
\begin{align*}
\frac{d}{dt}(\rho_1-\rho_2)&=2\kappa\tilde{\theta}(\rho_2)(\rho_1-\rho_2)+\kappa\sum_{k=2}^{n}\tilde{\theta}(\rho_k)(\rho_1-\rho_2)\\
&=\kappa(\rho_1-\rho_2)\left(2\tilde{\theta}(\rho_2)+\sum_{k=1}^{n}\tilde{\theta}(\rho_k)\right)\geq0.
\end{align*}
So, we have
\[
\rho_1(t)-\rho_2(t)\geq \rho_1(0)-\rho_2(0).
\]
This result directly implies the desired result.
\end{proof}

Using a similar argument, we can prove the following lemma.

\begin{lemma}
Suppose that the initial data $\rho(0)$ satisfy
\[
\rho_1(0)=\rho_2(0)=\cdots=\rho_m(0)>\rho_{m+1}(0)\geq\cdots\geq \rho_{n}(0),
\]
and let $\rho$ be a solution to \eqref{F-4-1} with initial data $\rho(0)$. Then, we have
\[
\rho_1(t)=\rho_2(t)=\cdots=\rho_m(t)\quad\forall~t\geq0
\]
and
\begin{align}\label{F-7}
\rho_1(t)-\max_{m<j\leq n}\rho_j(t)\geq\rho_1(0)-\max_{m<j\leq n}\rho_j(0).
\end{align}
\end{lemma}

From the previous lemma, we know that $\rho_i$ and $\rho_j$ cannot converge to same value when $i\leq m<j$. So, we can obtain the following theorem.

\begin{theorem}\label{T3.1}
Suppose that the initial data $\rho(0)$ satisfy
\[
m=\left|\mathcal{S}_M\right|,\quad \mathcal{S}_M=\left\{j:\rho_j(0)=\max_{1\leq k\leq n}\rho_k(0),\quad 1\leq j\leq n\right\} 
\]
and let $\rho$ be a solution to \eqref{F-4-1} with initial data $\rho(0)$. Then, we have
\[
\lim_{t\to\infty}\rho_j(t)=\frac{1}{m}\quad\forall j\in \mathcal{S}_M\quad\text{and}\quad \lim_{t\to\infty}\rho_k(t)=0\quad\forall j\not\in\mathcal{S}_M.
\]
\end{theorem}

\begin{proof}
Without loss of generality, we set
\[
\rho_1(0)=\rho_2(0)=\cdots=\rho_m(0)>\rho_{m+1}(0)\geq\cdots\geq \rho_{n}(0).
\]
From the previous lemma, we know that
\[
\rho_1(t)=\rho_2(t)=\cdots=\rho_m(t)\quad\forall~t\geq0
\]
From \eqref{F-7}, we know that
\[
\lim_{t\to\infty}\rho_1(t)-\rho_j(t)\neq0
\]
for all $j>m$. From the form of equilibrium, the limit of $\rho_j(t)$ should be zero for all $j>m$. Finally, we can obtain
\[
\lim_{t\to\infty}\rho_j(t)=\frac{1}{m}
\]
for all $1\leq j\leq m$.
\end{proof}
\begin{remark}
\noindent$\bullet$ The initial data $\rho^0$ is given on $\mathcal{P}:=\{\rho: \rho_1+\cdots+\rho_{n}=1,\quad \rho_i\geq0\quad\forall~1\leq i\leq n\}$. If we assume that the initial data is randomly distributed on $\mathcal{P}$ uniformly then we know that the probability of $m\geq2$ is zero, since $\mathcal{P}$ is closed region of the $d+1$-dimensional plane and the case $m=k$ is a part of $n-k$-dimensional plane. Since $\mathrm{Prob}(m=1)=1$, for the generic initial data $\rho(0)$, the complete consensus exhibits. i.e. there exists only one index $1\leq j\leq n$ such that 
\[
\lim_{t\to\infty}\rho_j(t)=1\quad\text{and}\quad \lim_{t\to\infty}\rho_k(t)=0\quad\forall~k\neq j.
\]

\noindent$\bullet$ $\mathcal{E}_1$ is the set of stable equilibrium. Let $\rho^\infty:=(\rho_1^\infty,\rho_2^\infty, \cdots, \rho_{n}^\infty)=(1, 0, \cdots, 0)$ be an element of $\mathcal{E}_1$. Then, for any sufficiently small non-negative constants $\epsilon_2, \cdots, \epsilon_{n}$, we know that $(\rho_1, \rho_2, \cdots, \rho_{n})=(1-\epsilon_2-\cdots-\epsilon_{n}, \epsilon_2, \cdots, \epsilon_{n})$ converges to $\rho^\infty$. This implies that any state $\rho^\infty\in\mathcal{E}_1$ is a stable state.\\

\noindent$\bullet$ $\bigcup_{m=2}^{n}\mathcal{E}_m$ is the set of unstable equilibrium. Let $2\leq k\leq n$ and $\rho^\infty:=(\rho_1^\infty,\rho_2^\infty, \cdots, \rho_{n}^\infty)=\Big(\underbrace{\frac{1}{k}, \cdots, \frac{1}{k}}_{k\text{ times}}, \underbrace{0, \cdots, 0}_{d-k+2 \text{ times}}\Big)$ be an element of $\mathcal{E}_k$. For any constants $0<\epsilon<\frac{1}{k}$, $(\rho_1^\infty-\epsilon, \rho_2^\infty, \cdots, \rho_{d+1}^\infty, \rho_{n}^\infty+\epsilon)$ does not converges to $\rho^\infty$. This implies that any state $\rho^\infty\in\bigcup_{m=2}^{n}\mathcal{E}_m$ is an unstable state.
\end{remark}

Now, we investigate the convergence rate of system when $\theta(\rho_i, \rho_j)=\Big(\min(\rho_i, \rho_j)\Big)^\alpha$ for some $\alpha\geq1$. Obviously, this function $\theta_{ij}$ satisfies the suggested assumptions. We assume that for a given initial configuration $\rho^0$, the corresponding solution converges to $\rho^\infty\in\mathcal{E}_1$. Without loss of generality, we assume that
\[
\lim_{t\to\infty}\rho_1(t)=1\quad\text{and}\quad \lim_{t\to\infty}\rho_j(t)=0\quad\forall ~2\leq j\leq n.
\]
Then we know that 
\[
\rho_1(t)=\max_{1\leq i\leq n} \rho_i(t),\quad \rho_1(t)>\frac{1}{n}\quad\forall~t\geq0,
\]
and this yields the following calculation:
\begin{align}
\begin{aligned}\label{F-8}
\frac{d}{dt}\rho_1&=\sum_{k=1}^{n}\rho_k^\alpha(\rho_1-\rho_k)
\leq\sum_{k=1}^{n}(1-\rho_1)^\alpha(\rho_1-\rho_k)\\
&=(1-\rho_1)^\alpha((n)\rho_1-1)
\leq(d+1)(1-\rho_1)^\alpha.
\end{aligned}
\end{align}
Now, we have
\begin{align*}
\frac{d}{dt}\rho_1&=\sum_{k=1}^{n}\rho_k^\alpha(\rho_1-\rho_k)
\geq\left(\rho_1-\max_{2\leq j\leq n}\rho_j\right)\sum_{k=2}^{n}\rho_k^\alpha.
\end{align*}
From the H\"{o}lder inequality, we get
\[
\left(\sum_{k=2}^{n}1^{\alpha_*}\right)^{\frac{1}{\alpha_*}}\left(\sum_{k=2}^{n}\rho_k^\alpha\right)^{\frac{1}{\alpha}}\geq \left(\sum_{k=2}^{n}\rho_k\right),
\]
where $\alpha_*$ is the conjugate exponent of $\alpha$, i.e. $\alpha_*=\frac{\alpha}{\alpha-1}$. So, we have
\[
\sum_{k=2}^{n}\rho_k^\alpha\geq\frac{1}{(d+1)^{\alpha-1}}\left(\sum_{k=2}^{n}\rho_k\right)^\alpha=(d+1)^{1-\alpha}(1-\rho_1)^\alpha.
\]
Finally, we get
\begin{align}
\begin{aligned}\label{F-9}
\frac{d}{dt}\rho_1&\geq \left(\rho_1-\max_{2\leq j\leq n}\rho_j\right)(d+1)^{1-\alpha}(1-\rho_1)^\alpha\\
&\geq\left(\rho_1^0-\max_{2\leq j\leq n}\rho_j^0\right)(d+1)^{1-\alpha}(1-\rho_1)^\alpha.
\end{aligned}
\end{align}
Here, we use Lemma \ref{L3.4} in the last inequality. Now, we combine \eqref{F-8} and \eqref{F-9} to get
\begin{align}\label{F-10}
C_1(1-\rho_1)^\alpha\leq \frac{d}{dt}\rho_1\leq C_2(1-\rho_1)^\alpha
\end{align}
for some positive constants $C_1$ and $C_2$. From solving the ODE, we have the following lemma.
\begin{lemma}\label{L6.4}
Let $x(t)$ be a solution to the following ODE:
\[
\dot{x}(t)=C(1-x(t))^\alpha,\quad x(0)=x^0.
\]
for some positive constant $C$ and $\alpha\geq1$. Then, we have
\[
x(t)=\begin{cases}
1-(1-x^0)e^{-Ct}&\quad\text{when }\alpha=1,\\
\displaystyle1-\frac{1}{\left((1-x^0)^{1-\alpha}+C(\alpha-1)t\right)^{\frac{1}{\alpha-1}}}&\quad\text{when }\alpha>1.
\end{cases}
\]
\end{lemma}

We combine \eqref{F-10} and Lemma \ref{L6.4} to get
\[
(1-\rho_1^0)e^{-C_1t}\leq 1-\rho_1\leq (1-\rho_1^0)e^{-C_2 t}\quad\text{when }\alpha=1
\]
and
\[
\frac{1}{\left((1-\rho_1^0)^{1-\alpha}+C_1(\alpha-1)t\right)^{\alpha-1}}\leq 1-\rho_1\leq\frac{1}{\left((1-\rho_1^0)^{1-\alpha}+C_2(\alpha-1)t\right)^{\alpha-1}}\quad\text{when }\alpha>1.
\]
So, we can conclude the result as follows: 
\begin{align}\label{F-10-1}
1-\rho_1(t)\simeq\begin{cases}
\text{blow up in a finite time}\quad&\text{when}\quad0<\alpha<1,\\
\text{converges to zero exponentially}&\text{when}\quad\alpha=1,\\
Ct^{(1-\alpha)^{-1}}&\text{when}\quad1<\alpha,\\
\end{cases}
\end{align}

For example, when $\alpha=1, 2, 3$, we have
\begin{align}\label{F-11}
1-\rho_1(t)\simeq e^{-Ct},\quad 1-\rho_1(t)\simeq \frac{1}{t},\quad 1-\rho_1(t)\simeq \frac{1}{\sqrt{t}},
\end{align}
respectively. So we can conclude that the convergence rate is sensitive to $\alpha\geq1$. So far, we investigated emergent behaviors on the complete graph. From now on, we consider the system \eqref{F-2-2} on general graphs. We set the network topology as follows: 
\begin{align}\label{F-12-0}
\omega_{ij}=\begin{cases}
1\quad\text{when }(i, j)\in E,\\
0\quad\text{when }(i, j)\not\in E.
\end{cases}
\end{align}
From a simple calculation, we can obtain the following system:
\begin{align}\label{F-12}
\frac{d\rho_j}{dt}=\kappa\sum_{l\in N_j}\theta_{jl}(\rho_j-\rho_l).
\end{align}
This yields
\begin{align*}
\frac{d}{dt}\sum_{j=1}^n\rho_j^2&=\sum_{j=1}^n2\kappa\rho_j\left(\sum_{l\in N_j}\theta_{jl}(\rho_j-\rho_l)\right)\\
&=2\kappa\sum_{(j, l)\in E}\theta_{jl}\rho_j(\rho_j-\rho_l)
=\kappa\sum_{(j, l)\in E}\theta_{jl}(\rho_j-\rho_l)^2\geq0.
\end{align*}
From the continuity of $\theta_{jl}$, we also apply Barbalat's lemma \cite{Bar} to get
\begin{align*}\label{F-13}
\lim_{t\to\infty}\sum_{(j, l)\in E}\theta_{jl}(\rho_j-\rho_l)^2=0.
\end{align*}
Since each term in the sum is non-negative, we again obtain
\[
\lim_{t\to\infty}(\rho_j-\rho_l)^2\theta_{jl}=0\quad\forall~(j, l)\in E.
\]
From this relation, if $(j, l)\in E$ then we have either
\begin{align}\label{F-14}
\lim_{t\to\infty}(\rho_j-\rho_l)=0\quad\text{or}\quad \lim_{t\to\infty}\min(\rho_j, \rho_l)=0.
\end{align}
When the graph is complete(i.e. $E=(V\times V)\backslash \mathrm{diag}(V\times V)$), then \eqref{F-14} holds for all $j, l\in V$ with $j\neq l$. This yields that if there exists one index $i\in V$ such that $\rho_i(0)>\max_{j\neq i}\rho_j(0)$, then 
\begin{align}\label{F-15}
\rho_i\to 1,\quad \rho_j\to0\quad\forall~j\in V\backslash\{i\}.
\end{align}
We have already shown the above property in Theorem \ref{T3.1}. However, when the graph is not a complete graph, then we have the counter example for \eqref{F-15}. We consider the graph is given as a square. We denote four vertices as $A$, $B$, $C$ and $D$, counter-clockwise order.\\

We also assume that the initial density is given as
\[
\rho_A(0)=0.6,\quad \rho_B(0)=0.1,\quad \rho_C(0)=0.2,\quad \rho_D(0)=0.1.
\]
Then, the we can easily check that
\[
\rho_A(t)\text{ and } \rho_C(t)\quad\text{increase},\quad \rho_B(t)\text{ and }\rho_D(t)\quad\text{decrease}.
\]
Using \eqref{F-15}, we can conclude that
\[
\lim_{t\to\infty}\rho_A(t)=\rho_A^\infty,\quad \lim_{t\to\infty}\rho_C(t)=\rho_C^\infty
\]
for some $\rho_A^\infty>0.6$ and $\rho_C^\infty>0.2$ which satisfies $\rho_A^\infty+\rho_C^\infty=1$, and
\[
\lim_{t\to\infty}\rho_B(t)=\lim_{t\to\infty}\rho_D(t)=0.
\]
We can say that the final state $(\rho_A^\infty, 0, \rho_C^\infty, 0)$ corresponds to the bipolar state of the Kuramoto model. We can summarize the above results as the following proposition.
\begin{proposition}\label{P3.2}
Let $(\rho, S)$ be a solution to system \eqref{D-10}, and the network topology is given as \eqref{F-12-0}. If $(j, l)\in E$, then \eqref{F-14} holds.
\end{proposition}

\subsection{Second order dynamics on the two points graph}\label{sec:3.3}
In this subsection, we study emergent behavior of system \eqref{F-4} with general initial configuration when $n=2$. System \eqref{F-4} with $n=2$ can be written as 
\begin{equation*}
\left\{\begin{aligned}
\frac{d\rho_1}{dt}&=(S_1-S_2)\theta_{12},\\
\frac{d\rho_2}{dt}&=-(S_1-S_2)\theta_{12},\\
\frac{dS_1}{dt}&=\kappa^2(\rho_1-\rho_2)\theta_{12}+\frac{1}{2}(-(S_1-S_2)^2+\kappa^2(\rho_1-\rho_2)^2)\frac{\partial\theta_{12}}{\partial \rho_1},\\
\frac{dS_2}{dt}&=-\kappa^2(\rho_1-\rho_2)\theta_{12}+\frac{1}{2}(-(S_1-S_2)^2+\kappa^2(\rho_1-\rho_2)^2)\frac{\partial\theta_{12}}{\partial \rho_2}.
\end{aligned}\right.
\end{equation*}
This yields
\begin{equation*}
\left\{\begin{aligned}
&\frac{d}{dt}(\rho_1-\rho_2)=2(S_1-S_2)\theta_{12},\\
&\frac{d}{dt}(S_1-S_2)=2\kappa^2(\rho_1-\rho_2)\theta_{12}-((S_1-S_2)^2-\kappa^2(\rho_1-\rho_2)^2)\left(\frac{\partial\theta_{12}}{\partial \rho_1}-\frac{\partial\theta_{12}}{\partial\rho_2}\right).
\end{aligned}\right.
\end{equation*}
Now, we substitute $\rho_1=r$, $\rho_2=1-r$, $S=S_1-S_2$, and $\theta_{12}=\theta$ into the above system to get
\begin{equation}\label{F-15-1}
\left\{\begin{aligned}
&\frac{dr}{dt}=S\theta,\\
&\frac{dS}{dt}=2\kappa^2(2r-1)\theta-\left(S^2-\kappa^2(2r-1)^2\right)\frac{d\theta}{dr}.
\end{aligned}\right.
\end{equation}
Here we used
\[
\frac{d\theta}{dr}=\frac{\partial \theta_{12}}{\partial \rho_1}\frac{d\rho_1}{dr}+\frac{\partial \theta_{12}}{\partial \rho_2}\frac{d\rho_2}{dr}=\frac{\partial \theta_{12}}{\partial \rho_1}-\frac{\partial \theta_{12}}{\partial \rho_2}.
\]
Recall that the following Hamiltonian
\begin{align}\label{F-15-2}
\mathcal{H}=\frac{\theta}{2}(S^2-\kappa^2(2r-1)^2)
\end{align}
is a constant of motion. If $\mathcal{H}(0)=:\mathcal{H}_0=0$, then it can be reduced to a gradient flow formulation(subsection \ref{sec:3.2}). Let assume that $\mathcal{H}_0>0$. Then, we get
\[
S^2=\kappa^2(2r-1)^2+\frac{2\mathcal{H}_0}{\theta}>0.
\]
This implies that $S(t)\neq0$ for all $t$. From this fact, if we assume the initial data satisfies $S_0>0$ then $S(t)>0$ for all $t$. This yields
\[
S=\sqrt{\kappa^2(2r-1)^2+\frac{2\mathcal{H}_0}{\theta}}.
\]
Since $S\theta>0$ for all $t$, we know that $r$ increases. Also, we know that $S\geq \sqrt{\frac{2\mathcal{H}_0}{\theta_0}}>0$. This implies that 
\[
\lim_{t\to\infty}r(t)=1.
\]

Now, we are interested in the convergence rate of $r$. We can easily deduce that $S$ tends to infinity as $r\to 1$, since $\theta\big|_{r=1}=0$. Since $\kappa^2(2r-1)^2$ is bounded, we can say that 
\[
S\simeq \sqrt{\frac{2\mathcal{H}_0}{\theta}}.
\]
Then, we have
\[
\frac{dr}{dt}=S\theta\simeq \sqrt{2\mathcal{H}_0\theta}.
\]
Since $r$ converges to one, we assume have that $\min(r, 1-r)=1-r$. We assume that $\theta\simeq (1-r)^\alpha$ for some $\alpha$. This yields
\[
\frac{\dot{r}}{(1-r)^{\alpha/2}}\simeq \sqrt{2\mathcal{H}_0}
\]
If $\alpha=2$, we get
\[
\ln(1-r_0)-\ln(1-r(t))\simeq \sqrt{2\mathcal{H}_0}t,
\]
and this implies
\[
1-r(t)\simeq(1-r_0)e^{-\sqrt{2\mathcal{H}_0}t}.
\]
If $\alpha>2$, then we have
\[
\left[-\left(1-r\right)^{1-\alpha/2}\right]_0^t\simeq \left(\frac{\alpha}{2}-1\right)\sqrt{2\mathcal{H}_0}t.
\]
This yields
\[
(1-r(t))^{1-\alpha/2}\simeq(1-r_0)^{1-\alpha/2}+ \left(\frac{\alpha}{2}-1\right)\sqrt{2\mathcal{H}_0}t,
\]
or equivalently
\[
1-r(t)\simeq\frac{1}{\left((1-r_0)^{1-\alpha/2}+ \left(\frac{\alpha}{2}-1\right)\sqrt{2\mathcal{H}_0}t\right)^{(\alpha/2-1)^{-1}}}\simeq \frac{C}{t^{(\alpha/2-1)^{-1}}}.
\]
If $0<\alpha<2$, then the solution blow up in a finite time. We can summarize the above results as follows:
\begin{align}\label{F-15-3}
1-r(t)\simeq\begin{cases}
\text{blow up in a finite time}\quad&\text{when}\quad 0<\alpha<2,\\
Ce^{-\sqrt{2\mathcal{H}_0}t}&\text{when}\quad \alpha=2,\\
Ct^{(1-\alpha/2)^{-1}}\quad&\text{when}\quad 2< \alpha,
\end{cases}\quad\text{as}\quad t\to\infty.
\end{align}

We can summarize the above result below. 
\begin{proposition}
Let $(r, S)$ be a solution of \eqref{F-15-1} with the initial data $(r_0, S_0)$ which satisfies the following:
\[
\mathcal{H}_0=\mathcal{H}(0)>0,\quad S_0>0,
\]
where $\mathcal{H}$ is defined in \eqref{F-15-2}. Then, $r(t)$ satisfies \eqref{F-15-3}.
\end{proposition}

\begin{remark}
We can compare two results \eqref{F-10-1} and \eqref{F-15-3}. The case \eqref{F-10-1} occurs when $\mathcal{H}_0=0$ and the case \eqref{F-15-3} occurs when $\mathcal{H}_0>0$. So, we can check the existence of bifurcation at $\mathcal{H}_0=0$.
\end{remark}
We will check general case $n>2$ in the numerical example section (Section \ref{sec:5.3}).

\subsection{Hopf-Cole transform on discrete graphs}\label{sec:3.4}
In this subsection, we reduce the following system \eqref{D-2} onto discrete graphs and provide its Hopf-Cole transform. For the notation simplicities, we set
\[
\partial_k\mathcal{F}=\frac{\partial\mathcal{F}}{\partial\rho_k},\quad \partial^2_{kl}\mathcal{F}=\frac{\partial^2\mathcal{F}}{\partial\rho_k\partial\rho_l}\quad \forall~1\leq k, l\leq n.
\]
From a similar argument which used in \cite{C-L-Z}, we can discretize system \eqref{D-2} as follows:
\begin{align}\label{D-10}
\begin{cases}
\displaystyle \frac{d\rho_j}{dt}-\sum_{l\in N_j}\omega_{jl}(S_j-S_l)\theta_{jl}=0,\vspace{0.2cm}\\
\displaystyle \frac{dS_j}{dt}+\frac{1}{2}\sum_{l\in N_j}\omega_{jl}(S_j-S_l)^2\frac{\partial \theta_{jl}}{\partial \rho_j}=\frac{\partial}{\partial\rho_j}\left(\frac{1}{2}\sum_{(k, l)\in E}\frac{\omega_{kl}}{2}(\partial_k\mathcal{F}-\partial_l\mathcal{F})^2\theta_{kl}\right).
\end{cases}
\end{align}
Also, the initial condition of continuous domain version system \eqref{D-2}
\[
S_0=-\frac{\delta \mathcal{F}(\rho_0)}{\delta\rho_0}
\]
 can be reduced as follows:
\begin{align}\label{D-11}
S_j(0)=-\partial_j\mathcal{F}(0):=-\frac{\partial\mathcal{F}(\rho)}{\partial\rho_j}\Bigg|_{\rho=\rho(0)},\quad \rho=(\rho_1, \cdots, \rho_n).
\end{align}
Here, $\mathcal{F}$ is a kind of potential functional of the system. When  the initial condition of system \eqref{D-10} is given as \eqref{D-11}, then the system can be reduce to the gradient flow system with potential $\mathcal{F}$. On the other hand, if the initial condition satisfies
\begin{align}\label{D-11-0}
S_j(0)=-\partial_j\mathcal{F}(0):=-\frac{\partial\mathcal{F}(\rho)}{\partial\rho_j}\Bigg|_{\rho=\rho(0)},\quad \rho=(\rho_1, \cdots, \rho_n).
\end{align}
then the system can be reduced to the gradient flow system with potential $-\mathcal{F}$. This phenomena can be interpreted as that initial conditions \eqref{D-11} and \eqref{D-11-0} yield a gradient descent flow and a gradient upscent flow, respectively.\\

We can also express system \eqref{D-10} as follows:
\begin{align}\label{D-11-1}
\begin{cases}
\displaystyle\frac{d}{dt}\rho_j=\frac{\partial}{\partial S_j}\mathcal{H}(\rho, S),\vspace{0.2cm}\\
\displaystyle\frac{d}{dt}S_j=-\frac{\partial}{\partial \rho_j}\mathcal{H}(\rho, S),
\end{cases}
\end{align}
where the Hamiltonian is given as follows:
\[
\mathcal{H}(\rho, S)=\frac{1}{4}\sum_{(i, j)\in E}\omega_{ij}\theta_{ij}\Big((S_i-S_j)^2-(\partial_i\mathcal{F}-\partial_j\mathcal{F})^2  \Big).
\]
So we can conclude that the discretized system \eqref{D-10} preserve the property of Hamiltonian systems, and this means the system is well-discretized. From relation \eqref{D-11-1}, we can easily check that
\[
\frac{d}{dt}\mathcal{H}(\rho(t), S(t))=0
\]
when $(\rho(t), S(t))$ is a solution to system \eqref{D-10}. \\

Now, we want to obtain the discretized version of the generalized Hopf-Cole transformation \eqref{D-3-1}. To mimic substitution \eqref{D-3-0} of the continuous version, we consider the the following substitution:
\begin{align}\label{D-11-2}
\begin{cases}
\partial_j\mathcal{F}=\xi_j+\xi^*_j,\\
S_j=\xi_j-\xi^*_j,
\end{cases}\quad\Longleftrightarrow\qquad
\begin{cases}
\xi_j=\frac{1}{2}(\partial_j\mathcal{F}+S_j),\\
\xi^*_j=\frac{1}{2}(\partial_j\mathcal{F}-S_j).
\end{cases}
\end{align}
Then, we have the following proposition.
\begin{proposition}[Hopf-Cole transform on a finite graph]
Let $(\rho, S)$ be a solution to \eqref{D-10}, and the substitution $(\xi, \xi^*)$ is given as \eqref{D-11-2}.  Then $(\xi, \xi^*)$ follows the following dynamics:
\begin{align}\label{D-12}
\begin{cases}
\displaystyle\frac{d\xi_j}{dt}=\sum_{(k, l)\in E}\partial^2_{jk}\mathcal{F}\omega_{kl}\theta_{kl}(\xi_k-\xi_l)
+\sum_{l\in N_j}\omega_{jl}(\xi_l^*-\xi_j^*)(\xi_l-\xi_j)\frac{\partial \theta_{jl}}{\partial \rho_j},\\
\displaystyle\frac{d\xi^*_j}{dt}=-\sum_{(k, l)\in E}\partial^2_{jk}\mathcal{F}\omega_{kl}\theta_{kl}(\xi_k^*-\xi_l^*)
-\sum_{l\in N_j}\omega_{jl}(\xi_l^*-\xi_j^*)(\xi_l-\xi_j)\frac{\partial \theta_{jl}}{\partial \rho_j}.
\end{cases}
\end{align}
\end{proposition}

\begin{proof}
From \eqref{D-10} and \eqref{D-11-2}, we have the following calculation:
\begin{align*}
\frac{d\xi_j}{dt}&=\frac{1}{2}\frac{d\partial_j\mathcal{F}}{dt}+\frac{1}{2}\frac{dS_j}{dt}\\
&=\frac{1}{2}\sum_k\left(\partial^2_{jk}\mathcal{F}\frac{d\rho_k}{dt}\right)+\frac{1}{2}\frac{dS_j}{dt}\\
&=\frac{1}{2}\sum_{\substack{k,\\l\in N_k}}\partial^2_{jk}\mathcal{F}\omega_{kl}(S_k-S_l)\theta_{kl}-\frac{1}{4}\sum_{l\in N_j}\omega_{jl}(S_j-S_l)^2\frac{\partial \theta_{jl}}{\partial \rho_j}\\
&\quad+\frac{1}{4}\sum_{(k, l)\in E}\frac{\omega_{kl}}{2}(\partial_k\mathcal{F}-\partial_l\mathcal{F})^2\frac{\partial\theta_{kl}}{\partial \rho_j}+\frac{1}{4}\sum_{(k, l)\in E}\omega_{kl}(\partial_k\mathcal{F}-\partial_l\mathcal{F})(\partial^2_{kj}\mathcal{F}-\partial^2_{lj}\mathcal{F})\theta_{kl}\\
&=\frac{1}{2}\sum_{(k, l)\in E}\partial^2_{jk}\mathcal{F}\omega_{kl}\theta_{kl}(S_k-S_l)+\frac{1}{2}\sum_{(k, l)\in E}\partial^2_{kj}\mathcal{F}\omega_{kl}\theta_{kl}(\partial_k\mathcal{F}-\partial_l\mathcal{F})\\
&\quad-\frac{1}{4}\sum_{l\in N_j}\omega_{jl}\Big((S_j-S_l)^2-(\partial_k\mathcal{F}-\delta_l\mathcal{F})^2\Big)\frac{\partial \theta_{jl}}{\partial \rho_j}\\
&=-\frac{1}{2}\sum_{(k, l)\in E}\partial^2_{jk}\mathcal{F}\omega_{kl}\theta_{kl}(-S_k-\partial_k\mathcal{F}+S_l+\partial_l\mathcal{F})\\
&\quad-\frac{1}{4}\sum_{l\in N_j}\omega_{jl}(S_j-\partial_j\mathcal{F}-S_l+\partial_l\mathcal{F})(S_j+\partial_j\mathcal{F}-S_l-\partial_l\mathcal{F})\frac{\partial \theta_{jl}}{\partial \rho_j}\\
&=\sum_{(k, l)\in E}\partial^2_{jk}\mathcal{F}\omega_{kl}\theta_{kl}(\xi_k-\xi_l)
+\sum_{l\in N_j}\omega_{jl}(\xi_l^*-\xi_j^*)(\xi_l-\xi_j)\frac{\partial \theta_{jl}}{\partial \rho_j}.
\end{align*}
From a similar calculation, we have
\begin{align*}
\frac{d\xi^*_j}{dt}&=\frac{1}{2}\frac{d\partial\mathcal{F}_j}{dt}-\frac{1}{2}\frac{dS_j}{dt}\\
&=\frac{1}{2}\sum_k\left(\partial^2_{jk}\mathcal{F}\frac{d\rho_k}{dt}\right)-\frac{1}{2}\frac{dS_j}{dt}\\
&=\frac{1}{2}\sum_{\substack{k,\\l\in N_k}}\partial^2_{jk}\mathcal{F}\omega_{kl}(S_k-S_l)\theta_{kl}+\frac{1}{4}\sum_{l\in N_j}\omega_{jl}(S_j-S_l)^2\frac{\partial \theta_{jl}}{\partial \rho_j}\\
&\quad-\frac{1}{4}\sum_{(k, l)\in E}\frac{\omega_{kl}}{2}(\partial_k\mathcal{F}-\partial_l\mathcal{F})^2\frac{\partial\theta_{kl}}{\partial \rho_j}-\frac{1}{4}\sum_{(k, l)\in E}\omega_{kl}(\partial_k\mathcal{F}-\partial_l\mathcal{F})(\partial^2_{kj}\mathcal{F}-\partial^2_{lj}\mathcal{F})\theta_{kl}\\
&=\frac{1}{2}\sum_{(k, l)\in E}\partial^2_{jk}\mathcal{F}\omega_{kl}\theta_{kl}(S_k-S_l)-\frac{1}{2}\sum_{(k, l)\in E}\partial^2_{kj}\mathcal{F}\omega_{kl}\theta_{kl}(\partial_k\mathcal{F}-\partial_l\mathcal{F})\\
&\quad+\frac{1}{4}\sum_{l\in N_j}\omega_{jl}\Big((S_j-S_l)^2-(\partial_k\mathcal{F}-\partial_l\mathcal{F})^2\Big)\frac{\partial \theta_{jl}}{\partial \rho_j}\\
&=-\frac{1}{2}\sum_{(k, l)\in E}\partial^2_{jk}\mathcal{F}\omega_{kl}\theta_{kl}(-S_k+\partial_k\mathcal{F}+S_l-\partial_l\mathcal{F})\\
&\quad+\frac{1}{4}\sum_{l\in N_j}\omega_{jl}(S_j-\partial_j\mathcal{F}-S_l+\partial_l\mathcal{F})(S_j+\partial_j\mathcal{F}-S_l-\partial_l\mathcal{F})\frac{\partial \theta_{jl}}{\partial \rho_j}\\
&=-\sum_{(k, l)\in E}\partial^2_{jk}\mathcal{F}\omega_{kl}\theta_{kl}(\xi_k^*-\xi_l^*)
-\sum_{l\in N_j}\omega_{jl}(\xi_l^*-\xi_j^*)(\xi_l-\xi_j)\frac{\partial \theta_{jl}}{\partial \rho_j}.
\end{align*}
So, we have the desired result.
\end{proof}

If the initial data $(\rho(0), S(0))$ of system \eqref{D-10} satisfies \eqref{D-11}, then the initial data $(\xi, \xi^*)$ of system \eqref{D-12} should satisfy
\begin{align}\label{D-13}
\xi_j(0)=0\quad\forall~j=1, 2, \cdots, n.
\end{align}
Here, $n$ is the number of vertices. 
\begin{lemma}\label{L3.1}
Let $(\xi, \xi^*)$ be a solution to system \eqref{D-12} with the initial data $(\xi(0), \xi^*(0))$. If the initial data satisfies \eqref{D-13}, we have
\[
\xi_j(0)=0\quad\forall~j=1, 2, \cdots, n,\quad t\geq0.
\]
Furthermore, $\xi^*$ follows the following dynamics:
\[
\frac{d}{dt}\xi_j^*=-\sum_{(k, l)\in E}\partial^2\mathcal{F}_{jk}\theta_{kl}(\xi_k^*-\xi_l^*).
\]
\end{lemma}

\begin{proof}
From the first equation of system \eqref{D-12}, we can express 
\[
\frac{d\xi}{dt}(t)=A(t)\xi(t)
\]
for some time-dependent matrix of size $n\times n$. From the uniqueness of system, we can obtain the desired result.
\end{proof}

\begin{remark}
If we discretize system \eqref{D-10}, then such numeric scheme may not preserve $S_j+\delta\mathcal{F}_j$ due to the numeric errors. However, if we use \eqref{D-12} to discretize system, then $\xi_j(0)\equiv 0$ for all $j$ guarantee that $\xi_j\equiv0$ for all time. 
\end{remark}
Corresponding lemma for Lemma \ref{L3.1} on continuous domain can be organized as Lemma \ref{LA.3}.

\section{Two points graph: analytical solutions. }\label{sec:4}
\setcounter{equation}{0}
In this section, we study system \eqref{D-10} with $n=2$ analytically. 

System \eqref{D-10} with $n=2$ can be written as
\begin{align}\label{F-17-0}
\begin{cases}
\displaystyle\frac{d\rho_1}{dt}-(S_1-S_2)\theta_{12}=0,\vspace{0.2cm}\\
\displaystyle\frac{d\rho_2}{dt}-(S_2-S_1)\theta_{12}=0,\vspace{0.2cm}\\
\displaystyle\frac{dS_1}{dt}+\frac{1}{2}(S_1-S_2)^2\frac{\partial \theta_{12}}{\partial \rho_1}=\frac{\partial}{\partial\rho_1}\left(\frac{1}{2}(\partial_1\mathcal{F}-\partial_2\mathcal{F})^2\theta_{12}\right),\vspace{0.2cm}\\
\displaystyle\frac{dS_2}{dt}+\frac{1}{2}(S_1-S_2)^2\frac{\partial\theta_{12}}{\partial\rho_2}=\frac{\partial}{\partial\rho_2}\left(\frac{1}{2}(\partial_1\mathcal{F}-\partial_2\mathcal{F})^2\theta_{12}\right).
\end{cases}
\end{align}
Recall that the Hamiltonian of system \eqref{F-17-0} can be expressed as
\[
\mathcal{H}(\rho_1, \rho_2, S_1, S_2)=\frac{1}{2}\theta_{12}((S_1-S_2)^2-(\partial_1\mathcal{F}-\partial_2\mathcal{F})^2).
\]
If we substitute 
\[
\rho_1=r,\quad \rho_2=1-r,\quad S-1-S_2=S,\quad \theta_{12}=\theta,
\]
then system \eqref{F-17-0} yields the following dynamics of $r$ and $S$:

\[
\begin{cases}
\displaystyle\frac{dr}{dt}-S\theta=0,\vspace{0.2cm}\\
\displaystyle\frac{dS}{dt}+\frac{1}{2}S^2\frac{d\theta}{dr}=\frac{d}{dr}\left(\frac{ \theta}{2}\left(\frac{d\mathcal{F}}{dr}\right)^2\right)
\end{cases}
\]
and the corresponding Hamiltonian is
\[
\mathcal{H}(r, S)=\frac{\theta}{2}\left(S^2-\left(\frac{d\mathcal{F}}{dr}\right)^2\right).
\]

The corresponding Lagrangian is
\[
\mathcal{L}(r, S)=\frac{\theta}{2}\left(S^2+\left(\frac{d\mathcal{F}}{dr}\right)^2\right).
\]

We first demonstrate that the Hamiltonian flow \eqref{F-17-0} has a variational formulation. 
\begin{proposition}
Consider the following Lagrangian action  minimization problem:
\begin{equation}\label{A}
\begin{aligned}
\displaystyle\mathcal{A}(r_0,r_1):=\displaystyle\text{minimize}\quad&\int_0^1\left(\frac{\dot{r}^2}{2\theta}+\frac{\theta}{2}\left(\frac{d\mathcal{F}}{dr}\right)^2\right)dt\\
\text{subject to}\quad&r(0)=r_0,\quad r(1)=r_1.
\end{aligned}
\end{equation}
The critical point system of minimization problem \eqref{A} satisfies the Hamiltonian flow \eqref{F-17-0}. 
\end{proposition}
\begin{proof}
Since $S=\frac{\dot{r}}{\theta}$, we can rewrite the Hamiltonian as follows:
\[
\mathcal{H}(r, \dot{r})=\frac{\dot{r}^2}{2\theta}-\frac{\theta}{2}\left(\frac{d\mathcal{F}}{dr}\right)^2,
\]
where we consider the above Hamiltonian as a function of $r$ and $\dot{r}$, since $\theta$ is a function of $r$. Since the Hamiltonian is constant along time-evolution, we assume that
\[
\frac{\dot{r}^2}{2\theta}-\frac{\theta}{2}\left(\frac{d\mathcal{F}}{dr}\right)^2=\mathcal{H}_0.
\]
For the notation simplicity, we set
\[
r(0)=r_0\quad\text{and}\quad r(1)=r_1.
\]
We introduce the following substitution:
\[
\frac{dx}{dr}=\frac{1}{\sqrt{\theta}}.
\]
Then, we get
\begin{align}\label{z-1-0}
\frac{d^2}{dt^2}x&=\frac{d}{dt}\left(\frac{dx}{dr}\dot{r}\right)=\frac{d}{dt}\frac{\dot{r}}{\sqrt{\theta}}
=\frac{d}{dr}\left(\frac{\theta}{2}\left(\frac{d\mathcal{F}}{dr}\right)^2\right)\sqrt{\theta}=\frac{d}{dx}\left(\frac{1}{2}\left(\frac{d\mathcal{F}}{dx}\right)^2\right)
\end{align}
From the definition of the action, we have
\begin{align}
\mathcal{A}&=\int_0^1\mathcal{L}(r, \dot{r})dt
=\int_0^1\left(\frac{\dot{r}^2}{2\theta}+\frac{\theta}{2}\left(\frac{d\mathcal{F}}{dr}\right)^2\right)dt
=\int_0^1\left(\frac{1}{2}\dot{x}^2+\frac{1}{2}\left(\frac{d\mathcal{F}}{dx}\right)^2\right)dt.
\end{align}
\end{proof}
By the following proposition, we can obtain the explicit form of the solution.
\begin{proposition}
If $\mathcal{F}$ satisfies 
\begin{align}\label{z-4-0}
\mathcal{F}(r)=\frac{1}{2}\left(\int_{1/2}^r\frac{1}{\sqrt{\theta(r_*)}}dr_*\right)^2,
\end{align}
then we have the following explicit form of the solution $r(t)$:
\begin{align}\label{z-4}
\int_{1/2}^{r(t)}\frac{dr_*}{\sqrt{\theta(r_*)}}=\frac{1}{\sinh 1}\left(\sinh(1-t)\int_{1/2}^{r_0}\frac{dr_*}{\sqrt{\theta(r_*)}}+\sinh t\int_{1/2}^{r_1}\frac{dr_*}{\sqrt{\theta(r_*)}}\right).
\end{align}
In addition, the least action minimizer satisfies 
\[
 \mathcal{A}(r_0, r_1)=\frac{\cosh1}{\sinh1}\left(\int_{1/2}^{r_0}\frac{dr_*}{\sqrt{\theta(r_*)}}-\int_{1/2}^{r_1}\frac{dr_*}{\sqrt{\theta(r_*)}}\right)^2+\frac{2\cosh 1-1}{\sinh 1}\int_{1/2}^{r_0}\frac{dr_*}{\sqrt{\theta(r_*)}}\int_{1/2}^{r_1}\frac{dr_*}{\sqrt{\theta(r_*)}}.
\]
\end{proposition}
\begin{proof}
From the definition of $x(r)$, we get
\[
\mathcal{F}(r)=\frac{1}{2}x(r)^2.
\]
Since $x$ solves \eqref{z-1-0} and $\mathcal{F}(r)=\frac{1}{2} x(r)^2$, we have the general form of $x(t)$ as follows:
\[
x(t)=a\cosh t+b\sinh t,
\]
where $a=x(0)$ and $b=\dot{x}(0)$. Since $x(0)$ and $x(1)$ are fixed, we have
\[
a=x(0),\quad b=\frac{x(1)-x(0)\cosh 1}{\sinh 1}.
\]
Then, we have
\[
x(t)=x(0)\cosh t+\left(\frac{x(1)-x(0)\cosh 1}{\sinh 1}\right)\sinh t=x(0)\frac{\sinh(1-t)}{\sinh 1}+x(1)\frac{\sinh t}{\sinh1}.
\]
Finally, we substitute 
\[
x(t)=\int_{1/2}^{r(t)}\frac{dr_*}{\sqrt{\theta(r_*)}}
\]
into the above equation to get the desired result.
From \eqref{z-4}, we know that
\[
\frac{\dot{r}}{\sqrt{\theta(r_*)}}=\frac{1}{\sinh 1}\left(-\cosh(1-t)\int_{1/2}^{r_0}\frac{dr_*}{\sqrt{\theta(r_*)}}+\cosh t\int_{1/2}^{r_1}\frac{dr_*}{\sqrt{\theta(r_*)}}\right)
\]
and
\[
\frac{d\mathcal{F}}{dr}=\frac{1}{\sqrt{\theta(r)}}\int^r_{1/2}\frac{dr_*}{\sqrt{\theta(r_*)}}=\frac{1}{\sinh 1\sqrt{\theta(r)}}\left(\sinh(1-t)\int_{1/2}^{r_0}\frac{dr_*}{\sqrt{\theta(r_*)}}+\sinh t\int_{1/2}^{r_1}\frac{dr_*}{\sqrt{\theta(r_*)}}\right).
\]
We substitute \eqref{z-4} into $\mathcal{A}$ as follows:
\begin{align*}
&\int_0^1\left(
\frac{\dot{r}^2}{2\theta}+\frac{\theta}{2}\left(\frac{d\mathcal{F}}{dr}\right)^2
\right)dt\\
&=\frac{1}{2\sinh^21}\int_0^1 \left(-\cosh(1-t)\int_{1/2}^{r_0}\frac{dr_*}{\sqrt{\theta(r_*)}}+\cosh t\int_{1/2}^{r_1}\frac{dr_*}{\sqrt{\theta(r_*)}}\right)^2dt\\
&+\frac{1}{2\sinh^21}\int_0^1\left(\sinh(1-t)\int_{1/2}^{r_0}\frac{dr_*}{\sqrt{\theta(r_*)}}+\sinh t\int_{1/2}^{r_1}\frac{dr_*}{\sqrt{\theta(r_*)}}\right)^2dt\\
&=\frac{1}{2\sinh^21}\left(\left(\int_{1/2}^{r_0}\frac{dr_*}{\sqrt{\theta(r_*)}}\right)^2+\left(\int_{1/2}^{r_1}\frac{dr_*}{\sqrt{\theta(r_*)}}\right)^2\right)\int_0^1(\cosh^2(1-t)+\sinh^2(1-t))dt\\
&-\frac{1}{\sinh^21}\int_0^1(\cosh(1-t)\cosh t-\sinh(1-t)\sinh t)\int_{1/2}^{r_0}\frac{dr_*}{\sqrt{\theta(r_*)}}\int_{1/2}^{r_1}\frac{dr_*}{\sqrt{\theta(r_*)}}\\
&=\frac{\sinh 2}{2\sinh^21}\left(\left(\int_{1/2}^{r_0}\frac{dr_*}{\sqrt{\theta(r_*)}}\right)^2+\left(\int_{1/2}^{r_1}\frac{dr_*}{\sqrt{\theta(r_*)}}\right)^2\right)
-\frac{\sinh 1}{\sinh^21}\int_{1/2}^{r_0}\frac{dr_*}{\sqrt{\theta(r_*)}}\int_{1/2}^{r_1}\frac{dr_*}{\sqrt{\theta(r_*)}}\\
&=\frac{\cosh 1}{\sinh1}\left(\left(\int_{1/2}^{r_0}\frac{dr_*}{\sqrt{\theta(r_*)}}\right)^2+\left(\int_{1/2}^{r_1}\frac{dr_*}{\sqrt{\theta(r_*)}}\right)^2\right)
-\frac{1}{\sinh1}\int_{1/2}^{r_0}\frac{dr_*}{\sqrt{\theta(r_*)}}\int_{1/2}^{r_1}\frac{dr_*}{\sqrt{\theta(r_*)}}\\
&=\frac{\cosh1}{\sinh1}\left(\int_{1/2}^{r_0}\frac{dr_*}{\sqrt{\theta(r_*)}}-\int_{1/2}^{r_1}\frac{dr_*}{\sqrt{\theta(r_*)}}\right)^2+\frac{2\cosh 1-1}{\sinh 1}\int_{1/2}^{r_0}\frac{dr_*}{\sqrt{\theta(r_*)}}\int_{1/2}^{r_1}\frac{dr_*}{\sqrt{\theta(r_*)}}.
\end{align*}
So, we can conclude that the 
\[
 \mathcal{A}(r_0, r_1)=\frac{\cosh1}{\sinh1}\left(\int_{1/2}^{r_0}\frac{dr_*}{\sqrt{\theta(r_*)}}-\int_{1/2}^{r_1}\frac{dr_*}{\sqrt{\theta(r_*)}}\right)^2+\frac{2\cosh 1-1}{\sinh 1}\int_{1/2}^{r_0}\frac{dr_*}{\sqrt{\theta(r_*)}}\int_{1/2}^{r_1}\frac{dr_*}{\sqrt{\theta(r_*)}}.
\]
We can easily check that $\mathcal{A}(r_0, r_1)\geq0$ for any $r_0$ and $r_1$ and equality only holds for $r_0=r_1=\frac{1}{2}$, since
\[
\cosh 1 x^2+(2\cosh 1-1)xy+\cosh 1 y^2\geq0
\]
and the equality only holds for $x=y=0$($\because D=(2\cosh 1-1)^2-4\cosh^21<0$). 
\end{proof}
\begin{proposition}
We formulate a divergence function on a two point graph:
\begin{equation*}
\begin{split}
    D(r_0, r_1)&=\mathcal{A}(r_0, r_1)-\frac{1}{2}\mathcal{A}(r_0, r_0)-\frac{1}{2}\mathcal{A}(r_1, r_1)\\
&=\frac{1}{2\sinh1}\left(\int_{r_0}^{r_1}\frac{dr_*}{\sqrt{\theta(r_*)}}\right)^2.
\end{split}
\end{equation*}
\end{proposition}
\begin{remark}
We remark that the above divergence function $D(r_0,r_1)$ is a generalization of the discrete Wasserstein metric; see related studies in AI inference problems \cite{sd}. In our selection of $\mathcal{F}$ in \eqref{z-4-0},  $D$ recovers a two point discrete Wasserstein distances derived in \cite{Mass}.   
\end{remark}

\begin{proof}
The proof follows from a direct calculation. 
\begin{align*}
D(r_0, r_1)&= \mathcal{A}(r_0, r_1)-\frac{1}{2}\mathcal{A}(r_0, r_0)-\frac{1}{2}\mathcal{A}(r_1, r_1)\\
&=\frac{\cosh1}{\sinh1}\left(\int_{1/2}^{r_0}\frac{dr_*}{\sqrt{\theta(r_*)}}-\int_{1/2}^{r_1}\frac{dr_*}{\sqrt{\theta(r_*)}}\right)^2+\frac{2\cosh 1-1}{\sinh 1}\int_{1/2}^{r_0}\frac{dr_*}{\sqrt{\theta(r_*)}}\int_{1/2}^{r_1}\frac{dr_*}{\sqrt{\theta(r_*)}}\\
&\quad-\frac{2\cosh 1-1}{2\sinh 1}\left[\left(\int_{1/2}^{r_0}\frac{dr_*}{\sqrt{\theta(r_*)}}\right)^2+\left(\int_{1/2}^{r_1}\frac{dr_*}{\sqrt{\theta(r_*)}}\right)^2\right]\\
&=\left(\frac{\cosh1}{\sinh1}-\frac{2\cosh 1-1}{2\sinh 1}\right)\left(\int_{1/2}^{r_0}\frac{dr_*}{\sqrt{\theta(r_*)}}-\int_{1/2}^{r_1}\frac{dr_*}{\sqrt{\theta(r_*)}}\right)^2\\
&=\frac{1}{2\sinh1}\left(\int_{1/2}^{r_0}\frac{dr_*}{\sqrt{\theta(r_*)}}-\int_{1/2}^{r_1}\frac{dr_*}{\sqrt{\theta(r_*)}}\right)^2.
\end{align*}
\end{proof}

Now, we introduce some examples $\mathcal{F}$ which exhibits aggregation phenomena. Recall that the Kuramoto type potential \eqref{F-3} is minimized when 
\[
\rho_i=\begin{cases}
1\quad\text{when }i=j,\\
0\quad\text{when }i\neq j
\end{cases}
\]
for some fixed index $1\leq j\leq n$. Entropy functionals satisfy similar property. So, we review various types of entropies.

\begin{example}\label{R3.2}
Let $\rho=(\rho_1, \cdots, \rho_n)$ be a probability distribution given on $n$-point graph. Then, we have the following entropies.\\

\noindent(1) Shannon entropy:
\begin{align*}
H(\rho):=-\sum_{i=1}^n \rho_i \log \rho_i.
\end{align*}

\noindent(2) R\'{e}nyi entropy:
\begin{align*}
H_\alpha(\rho)=\frac{1}{1-\alpha}\log\left(\sum_{i=1}^n\rho_i^\alpha\right),
\end{align*}
where $\alpha$ is a non-negative constant with $\alpha\neq1$. From the L'H$\hat{o}$spital theorem, we can easily check that R\'{e}nyi entropy converges to Shannon entropy as $\alpha$ goes to 1. i.e.
\[
\lim_{\alpha\to 1} H_\alpha(\rho)=H(\rho).
\]
Similarly, we can also check that
\begin{align*}
\lim_{\alpha\to \infty}H_\alpha(\rho)&=\lim_{\alpha\to\infty}\frac{1}{1-\alpha}\log\left(\sum_{i=1}^n\rho_i^\alpha\right)=\lim_{\alpha\to\infty}\frac{\sum_{i=1}^n \log(\rho_i) \rho_i^{\alpha}}{-\sum_{i=1}^n \rho_i^\alpha}\\
&=-\lim_{\alpha\to\infty}\frac{\sum_{i=1}^n\log(\rho_i)(\rho_i/\max_j \rho_j)^\alpha}{\sum_{i=1}^n(\rho_i/\max_j \rho_j)^\alpha}=-\log(\max_i \rho_i).
\end{align*}
This implies the R\'{e}nyi entropy converges to the min-entropy.
\\

\noindent(3) Tsallis entropy:
\begin{align*}
S_q(\rho)=\frac{1}{q-1}\left(1-\sum_{i=1}^n\rho_i^q\right),
\end{align*}
where $q\geq1$. From the L'H$\hat{o}$spital theorem, we can also check that Tsallis entropy converges to Shannon entropy as $q\searrow 1$.\\

From a simple calculation, we can also prove that both $H_\alpha(P)$ and $S_q(P)$ is minimized when one of $p_i$ is one and the others are zero. Since these properties are similar, we can expect that $\mathcal{F}$ can be replaced by either R\'{e}nyi entropy or Tsallis entropy. However, in this paper, we focus on $\mathcal{F}$ defined in \eqref{F-3}.
\end{example}

Now, we consider some entropies as $\mathcal{F}$ introduced in Example \ref{R3.2} as potentials. Since we are considering two points graphs, we set $\rho(r)=(r, 1-r)$ with $0\leq r\leq 1$. Using the entropies that introduced in Example \ref{R3.2}, we define the functional $\mathcal{F}$ which can be expressed \eqref{z-4-0} as follows:
\begin{align}\label{z-5}
\mathcal{F}(r)=\begin{cases}
-\Big(H(\rho(r))-H(\rho(1/2))\Big)=\log 2+r\log r+(1-r)\log(1-r)\quad&\text{Shannon entropy potential},\\
-\Big(H_\alpha(\rho(r))-H_\alpha(\rho(1/2))\Big)=\log 2-\frac{1}{1-\alpha}\log(r^\alpha+(1-r)^\alpha)\quad&\text{R\'{e}nyi entropy potential},\\
-\Big(S_q(\rho(r))-S_q(\rho(1/2))\Big)=\frac{1}{q-1}(r^q+(1-r)^q-2^{1-q})
\quad&\text{Tsallis entropy potential}.
\end{cases}
\end{align}

For the Shannon entropy potential, we subtract $H(\rho(1/2))$ from $H(\rho(r))$. Since $\mathcal{F}(r)$ is positive and $H(\rho(r))-H(\rho(1/2))$ is negative, we put the minus sign. Not only the Shannon entropy, we also applied this argument in R\'{e}nyi entropy and Tsallis entropy to obtain R\'{e}nyi entropy potential and Tsallis entropy potential, respectively. From \eqref{z-4-0} and \eqref{z-5}, we can also define $\theta$, respectively. Relation \eqref{z-4-0} yields
\[
2\mathcal{F}(r)=\left(\int_{1/2}^r\frac{1}{\sqrt{\theta(r_*)}}dr_*\right)^2.
\]
This implies that
\[
\int_{1/2}^r\frac{1}{\sqrt{\theta(r_*)}}dr_*=\begin{cases}
\displaystyle-\sqrt{2\mathcal{F}(r)}\quad&\displaystyle\text{if}\quad 0\leq r\leq \frac{1}{2},\\
\displaystyle\sqrt{2\mathcal{F}(r)}\quad&\displaystyle\text{if}\quad\frac{1}{2}\leq r\leq 1,
\end{cases}
\]
and this again yields
\[
\frac{1}{\sqrt{\theta(r)}}=\begin{cases}
\displaystyle-\frac{d}{dr}\sqrt{2\mathcal{F}(r)}\quad&\displaystyle\text{if}\quad 0\leq r\leq \frac{1}{2},\vspace{0.2cm}\\
\displaystyle\frac{d}{dr}\sqrt{2\mathcal{F}(r)}\quad&\displaystyle\text{if}\quad\frac{1}{2}\leq r\leq 1.
\end{cases}
\]
Finally, we can obtain that
\begin{align}\label{z-6}
\theta(r)=\left(\frac{d}{dr}\sqrt{2\mathcal{F}(r)}\right)^{-2}=\frac{2\mathcal{F}(r)}{\left(\frac{d}{dr}\mathcal{F}(r)\right)^2}.
\end{align}
We substitute \eqref{z-5} into \eqref{z-6} to obtain
\[
\theta(r)=\begin{cases}
\displaystyle\frac{2(\log 2+r\log r+(1-r)\log(1-r))}{\left(\log r-\log(1-r)\right)^2}&\quad \text{Shannon entropy},\\
\displaystyle\frac{2(1-\alpha)^2(r^\alpha+(1-r)^\alpha)^2\left(\log 2-\frac{1}{1-\alpha}\log(r^\alpha+(1-r)^\alpha)\right)}{\alpha^2\left(
r^{\alpha-1}-(1-r)^{\alpha-1}
\right)^2} &\quad\text{R\'{e}nyi entropy},\\
\displaystyle\frac{2(q-1)(r^q+(1-r)^q-2^{1-q})}{q^2\left(
r^{q-1}-(1-r)^{q-1}
\right)^2}&\quad \text{Tsallis entropy}.
\end{cases}
\]
The minimum action $\mathcal{A}$ can be expressed as $\mathcal{F}$ as follows:
\[
 \mathcal{A}(r_0, r_1)=
 \begin{cases}
\displaystyle \frac{2\cosh 1}{\sinh 1}\left(
\mathcal{F}(r_0)+\mathcal{F}(r_1)
 \right)+\frac{2}{\sinh 1}\sqrt{\mathcal{F}(r_0)\mathcal{F}(r_1)}\quad\text{if }r_0\leq\frac{1}{2}\leq r_1~~\text{or}~~r_1\leq \frac{1}{2}\leq r_0,\vspace{0.2cm}\\
\displaystyle  \frac{2\cosh 1}{\sinh 1}\left(
\mathcal{F}(r_0)+\mathcal{F}(r_1)
 \right)-\frac{2}{\sinh 1}\sqrt{\mathcal{F}(r_0)\mathcal{F}(r_1)}\quad\text{otherwise}.
 \end{cases} 
\]

We notice that $\mathcal{A}$ is the value function for discrete Schr{\"o}dinger bridge system \cite{CLMZ}. See also motivations in \cite{CGP}. 

\section{Numerical examples}\label{sec:5}
\setcounter{equation}{0}

In this section, we provide some numeric examples. This section is consisted of two parts. In the first part, we provide simulations of complete graphs. We check \eqref{F-11} numerically. In the second part, we consider general graphs instead of complete graphs. For this case, we have Proposition \ref{P3.2}. We will check this proposition again numerically on various discrete graphs.
\subsection{Complete graph}\label{sec:5.1}
In the first part of this subsection we fix parameters of \eqref{F-4-1} as follows:
\[
\Delta t=0.01,\quad \kappa=1,\quad n=4,\quad \rho^0=(0.5, 0.3, 0.15, 0.05).
\]
We will check \eqref{F-11} numerically. First, we consider the case of $\theta_{ij}=\min(\rho_i, \rho_j)$. Since $\rho_1^0$ is the unique maximum of $\{\rho_1^0, \rho_2^0, \rho_3^0, \rho_4^0\}$, $\rho_1(t)$ converges to 1 exponentially. To check the exponential convergence, we draw the graph of $\log(1-\rho_1(t))$. See Figure \ref{Fig1}.\\
\begin{figure}[h]
\includegraphics[width=7.5cm]{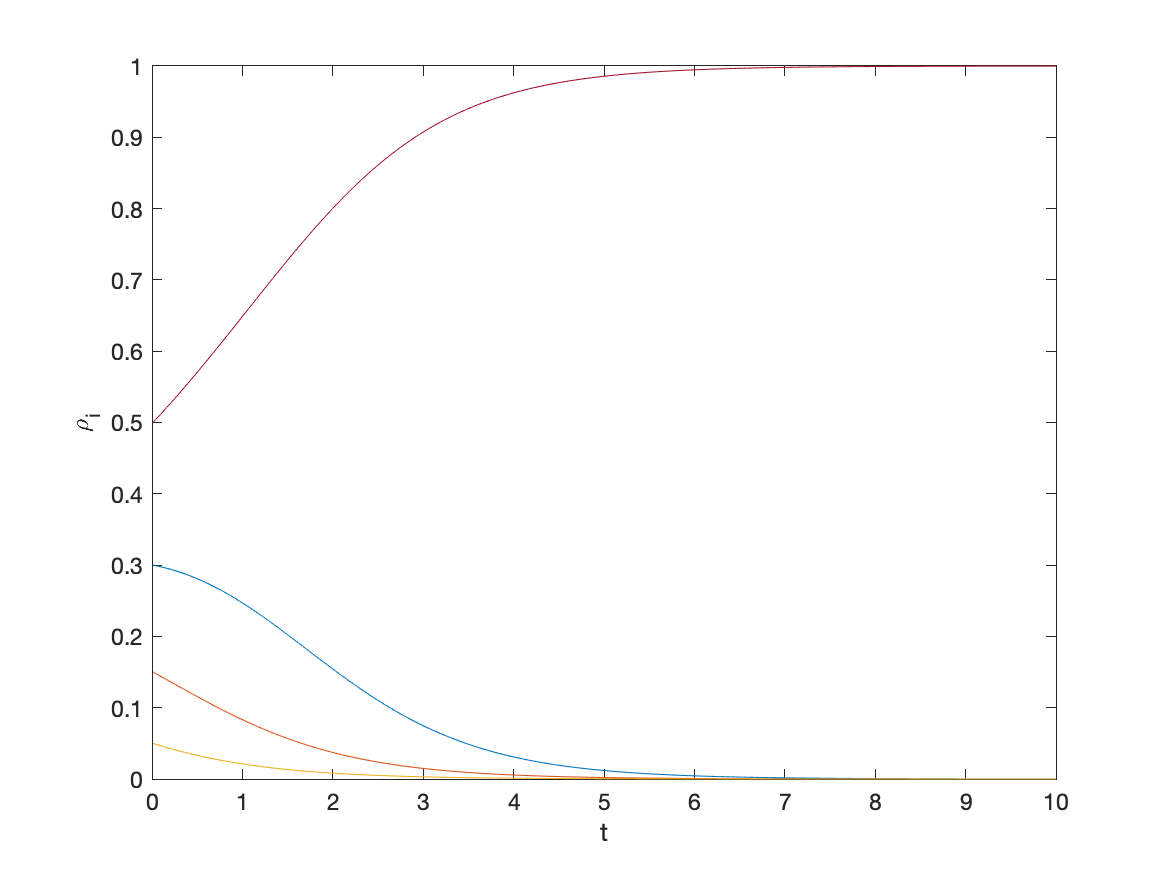}
\includegraphics[width=7.5cm]{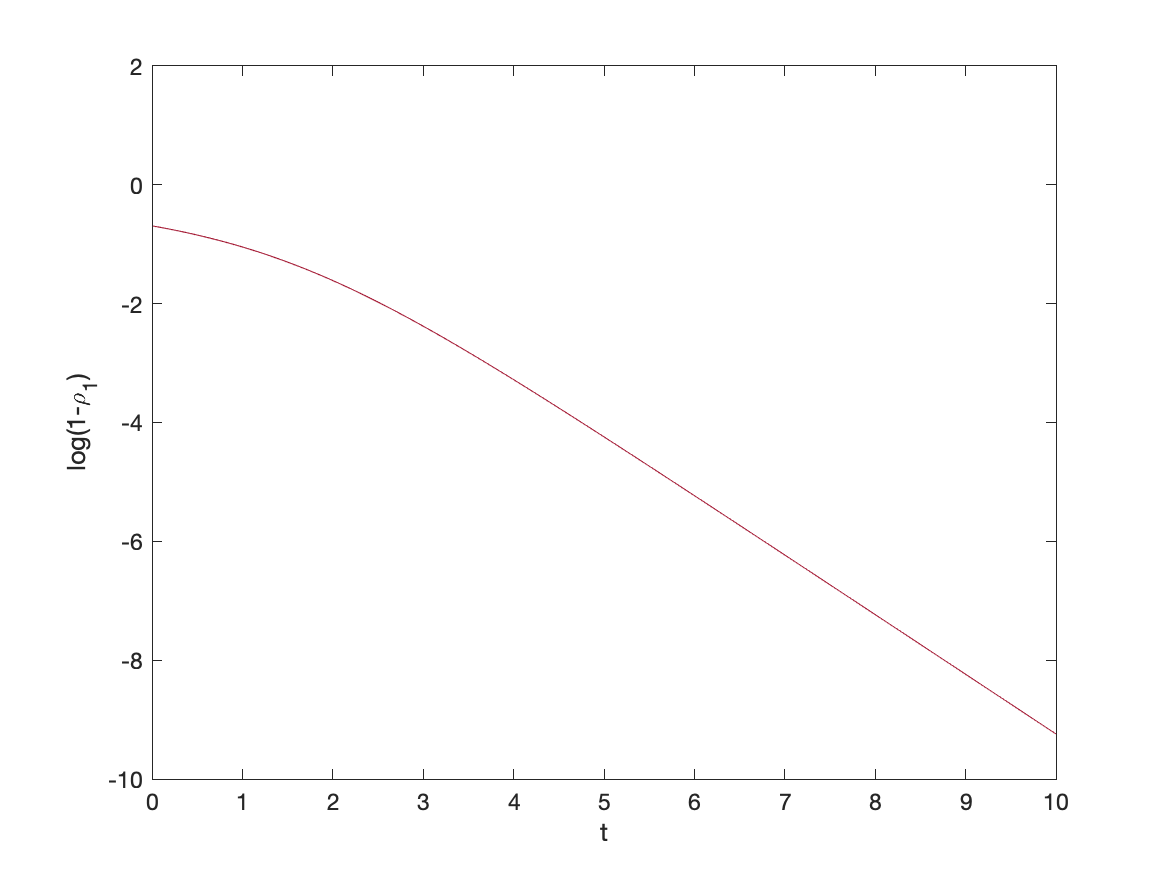}
\caption{The case of $\theta_{ij}=\min(\rho_i, \rho_j)$. As we expected $\log(1-\rho_1(t))$ converges to a line with a negative slope asymptotically as $t\to \infty$.  }\label{Fig1}
\end{figure}

Second, we consider the case of $\theta_{ij}=\min(\rho_i, \rho_j)^2$. We know that $\rho_1(t)$ converges to 1 with the following convergence rate: $1-\rho_1(t)\simeq\frac{1}{t}$. To check this convergence rate, we draw the graph of $(1-\rho_1(t))^{-1}$. See Figure \ref{Fig2}.\\

\begin{figure}[h]
\includegraphics[width=7.5cm]{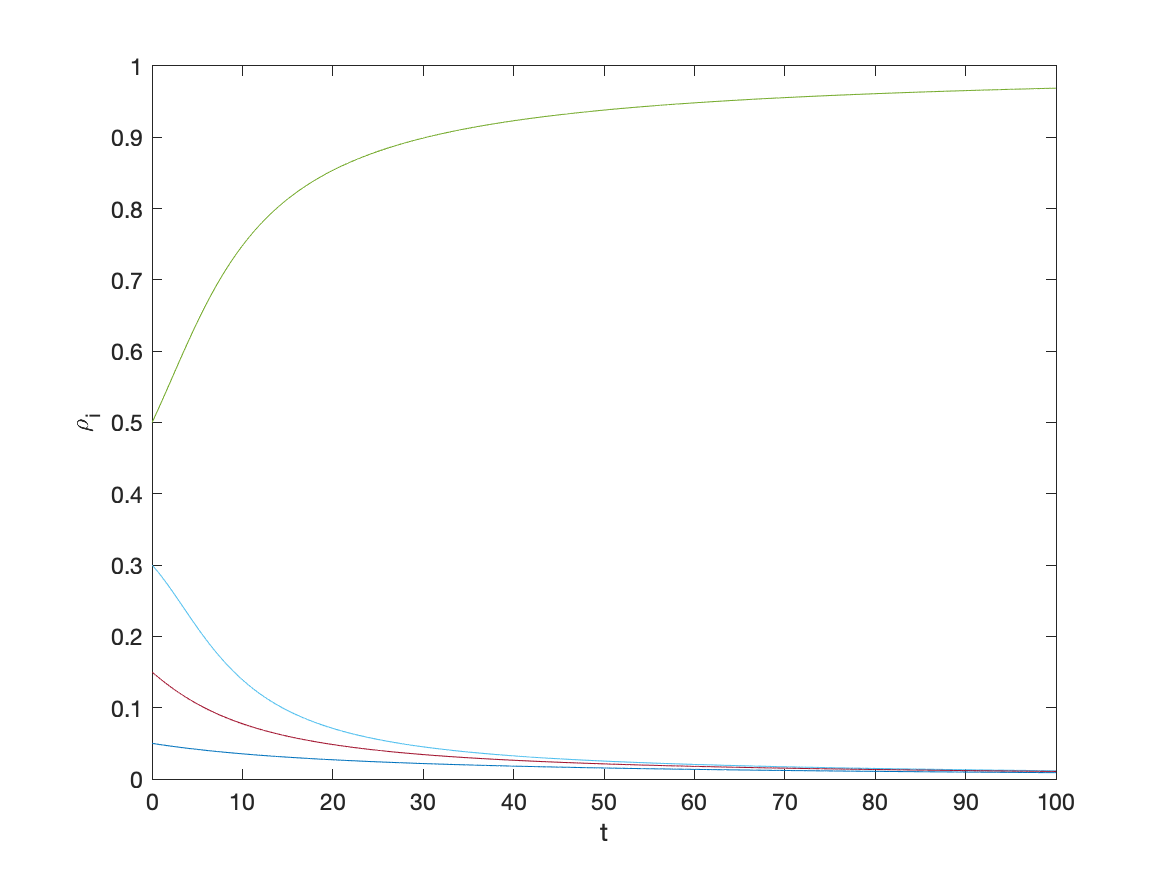}
\includegraphics[width=7.5cm]{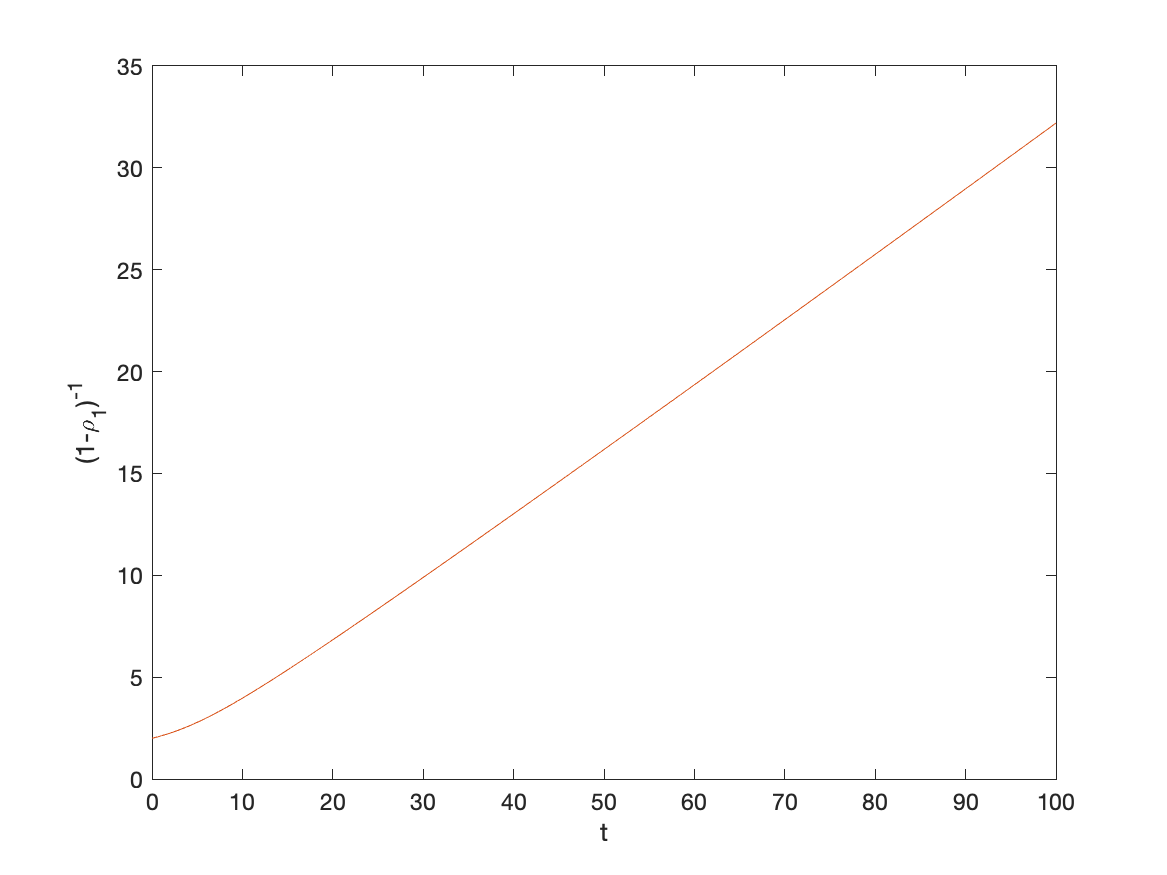}
\caption{The case of $\theta_{ij}=\min(\rho_i, \rho_j)^2$. As we expected the graph of $(1-\rho_1(t))^{-1}$ converges to a line with a positive slope asymptotically as $t\to\infty$. This implies that $\rho_1(t)\simeq \frac{1}{t}$.}\label{Fig2}
\end{figure}

Third, we consider the case of $\theta_{ij}=\min(\rho_i, \rho_j)^3$. We know that $\rho_1(t)$ converges to 1 with the following convergence rate: $1-\rho_1(t)\simeq\frac{1}{\sqrt{t}}$. To check this convergence rate, we draw the graph of $(1-\rho_1(t))^{-2}$. See Figure \ref{Fig3}.

\begin{figure}[h]
\includegraphics[width=7.5cm]{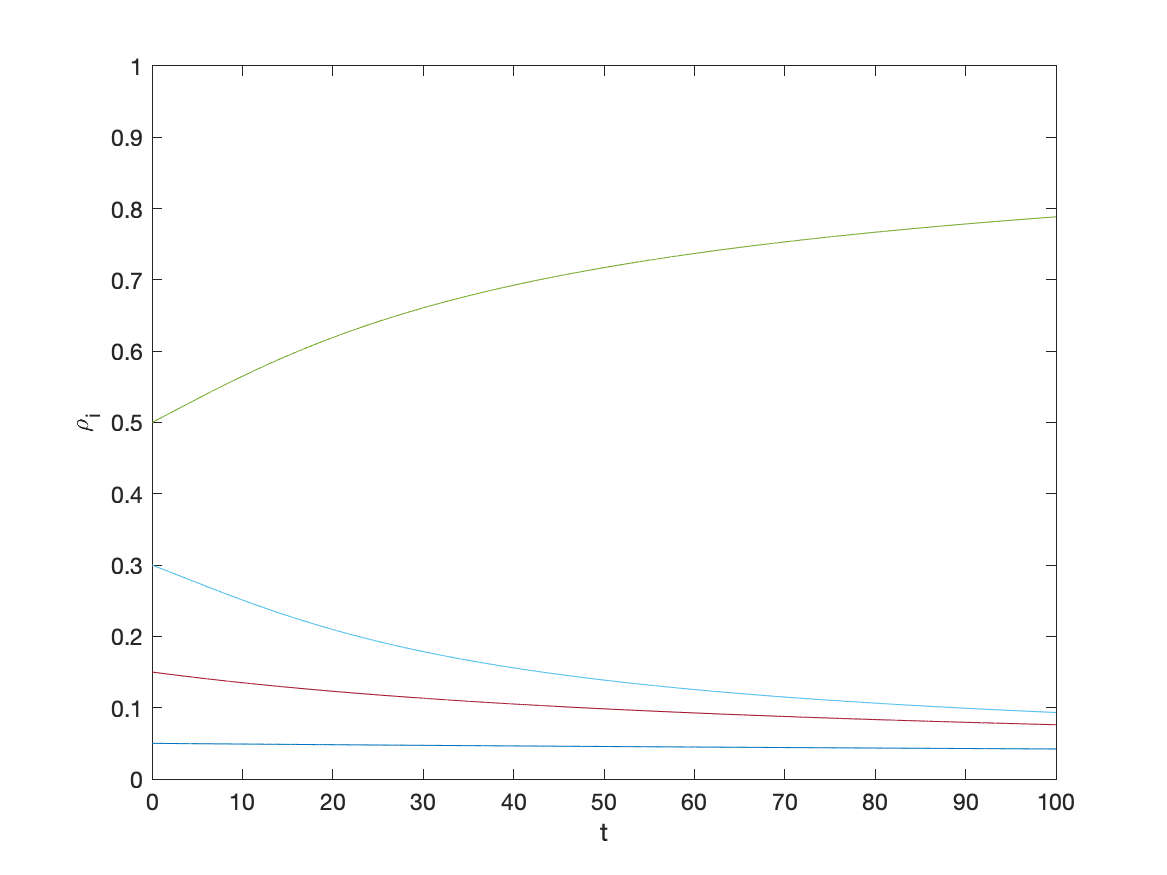}
\includegraphics[width=7.5cm]{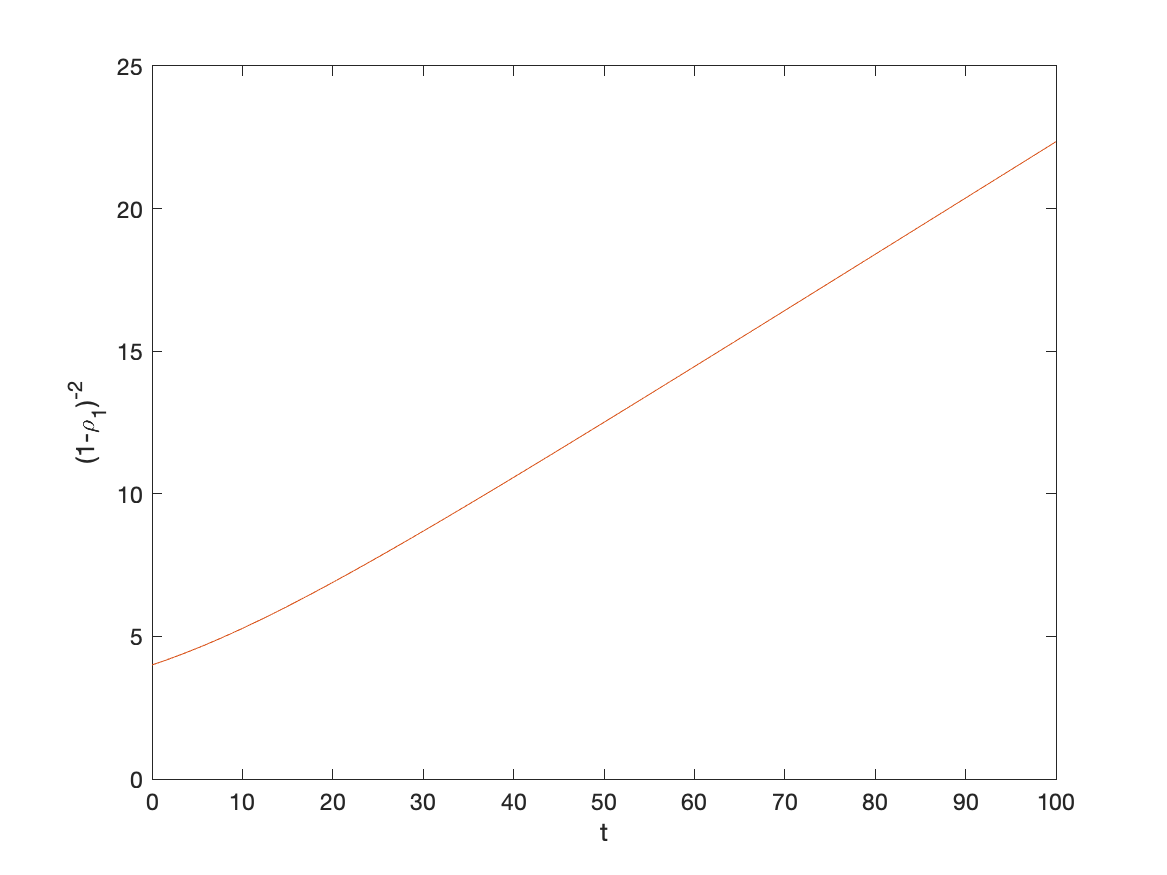}
\caption{The case of $\theta_{ij}=\min(\rho_i, \rho_j)^3$. As we expected the graph of $(1-\rho_1(t))^{-2}$ converges to a line with a positive slope asymptotically as $t\to\infty$. This implies that $\rho_1(t)\simeq \frac{1}{\sqrt{t}}$.}\label{Fig3}
\end{figure}

\subsection{General graph}\label{sec:5.2}
In the second part of this subsection, we provide some numeric examples of system \eqref{F-12} on general graphs. We consider three graphs containing six edges($A$, $B$, $C$, $D$, $E$, and $F$). (1) Cycle graph(Figure \ref{Fig4}), (2) Lattice graph(Figure \ref{Fig5}), and (3) Ribbon shaped graph(Figure \ref{Fig6}). Each graph contains six vertices and they have different topologies. In this part, we fix the following parameters:
\[
\Delta t=0.01,\quad \kappa=1,\quad \theta_{ij}=\min(\rho_i, \rho_j),\quad \rho(0)=(0.3, 0.2, 0.1, 0.1, 0.1 ,0.2),
\]
where $\rho=(\rho_A, \rho_B, \rho_C, \rho_D, \rho_E, \rho_F)$.

\begin{figure}[h]
\begin{tikzpicture} 
\filldraw (2, 0) circle (2pt) node[anchor=west] {A};
\filldraw (1, 1.732) circle (2pt) node[anchor=south]{B};
\filldraw (-1, 1.732) circle (2pt) node[anchor=south] {C};
\filldraw (-2, 0) circle (2pt) node[anchor=east] {D};
\filldraw (-1, -1.732) circle (2pt) node[anchor=north] {E};
\filldraw (1, -1.732) circle (2pt) node[anchor=north] {F};
\draw (2, 0)--(1, 1.732)--(-1, 1.732)--(-2, 0)--(-1, -1.732)--(1, -1.732)--(2, 0);
\end{tikzpicture}
\caption{Cycle graph. }\label{Fig4}
\end{figure}
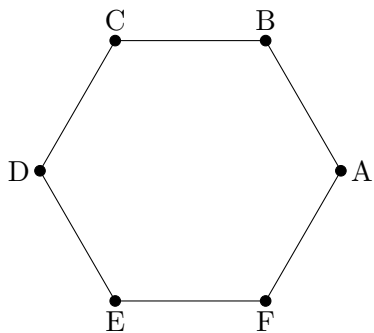

\begin{figure}[h]
\begin{tikzpicture} 
\filldraw (2, 0) circle (2pt) node[anchor=west] {A};
\filldraw (1, 1.732) circle (2pt) node[anchor=south]{B};
\filldraw (-1, 1.732) circle (2pt) node[anchor=south] {C};
\filldraw (-2, 0) circle (2pt) node[anchor=east] {D};
\filldraw (-1, -1.732) circle (2pt) node[anchor=north] {E};
\filldraw (1, -1.732) circle (2pt) node[anchor=north] {F};
\draw (2, 0)--(1, 1.732)--(-1, 1.732)--(-2, 0)--(-1, -1.732)--(1, -1.732);
\end{tikzpicture}
\caption{Lattice graph}\label{Fig5}
\end{figure}
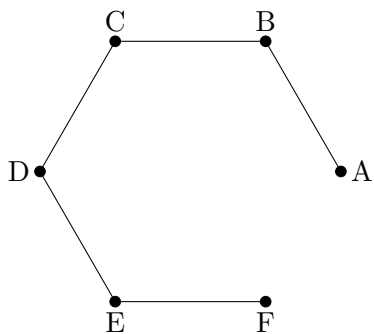

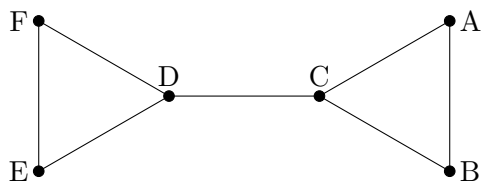
\begin{figure}[h]
\begin{tikzpicture}

\filldraw (1, 0) circle (2pt) node[anchor=south] {C};
\filldraw (1+1.732, 1) circle (2pt) node[anchor=west] {A};
\filldraw (1+1.732, -1) circle (2pt) node[anchor=west] {B};

\filldraw (-1, 0) circle (2pt) node[anchor=south] {D};
\filldraw (-1-1.732, 1) circle (2pt) node[anchor=east] {F};
\filldraw (-1-1.732, -1) circle (2pt) node[anchor=east] {E};

\draw (1, 0)--(1+1.732, -1)--(1+1.732, 1)--(1, 0);
\draw (-1, 0)--(-1-1.732, -1)--(-1-1.732, 1)--(-1, 0);
\draw (-1, 0)--(1, 0);

\end{tikzpicture}
\caption{Ribbon shaped graph}\label{Fig6}
\end{figure}

\begin{example}[Cycle graph: Figure \ref{Fig4}]\label{E4.1}
In this case, the density $\rho(t)$ converges to the following density:
\[
\rho^\infty=(0.7398,\quad    0,\quad    0,\quad   0.2602,\quad    0,\quad   0).
\]
We can observe that $A$ and $D$ are not connected, and the only nonzero part of $\rho^\infty$ are $\rho_A^\infty$ and $\rho_D^\infty$. This result follows Proposition \ref{P3.2}.
\end{example}

\begin{example}[Lattice graph: Figure \ref{Fig5}]\label{E4.2}
If we remove one edge $AF$ from the cycle graph(Figure \ref{Fig4}), then we can obtain the lattice graph(Figure \ref{Fig5}). In this case, the density $\rho(t)$ converges to the following density:
\[
\rho^\infty=(0.5274,\quad    0, \quad   0, \quad   0.1958,\quad    0, \quad   0.2768).
\]
We can observe that $A$, $D$, and $F$ are not connected to each others, and the only nonzero part of $\rho^\infty$ are $\rho_A^\infty$, $\rho_D^\infty$, and $\rho_F^\infty$. This result follows Proposition \ref{P3.2}.
\end{example}

\begin{example}[Ribbon shaped graph: Figure \ref{Fig6}]\label{E4.3}
In this case, the density $\rho(t)$ converges to the following density:
\[
\rho^\infty=(0.5948,\quad    0, \quad   0, \quad   0 ,  \quad 0, \quad   0.4052).
\]
We can observe that $A$ and $F$ are not connected to each others, and the only nonzero part of $\rho^\infty$ are $\rho_A^\infty$ and $\rho_F^\infty$. This result follows Proposition \ref{P3.2}.
\end{example}

\subsection{Second order dynamics with $n\geq3$.}\label{sec:5.3}
In this subsection, we provide numeric results on second order dynamics with $n\geq3$. In this part, we fix the following parameters:
\[
\Delta t=0.01,\quad \kappa=1,\quad n=6,\quad \theta_{ij}=\min(\rho_i, \rho_j)^2.
\]

We consider the initial data:
\begin{align}\label{IN1}
\begin{cases}
\rho^0=(0.3224,~0.2108,~0.1071,~0.0713,~0.2518,~0.0366),\\
S^0=(0.1597,~-1.1129,~0.5929,~0.4568,~0.8299,~-0.2499).
\end{cases}
\end{align}
We present the numeric solution of system \eqref{F-4} with initial data \eqref{IN1} in Figure \ref{Fig7}.
\begin{figure}[h]
\includegraphics[width=7.5cm]{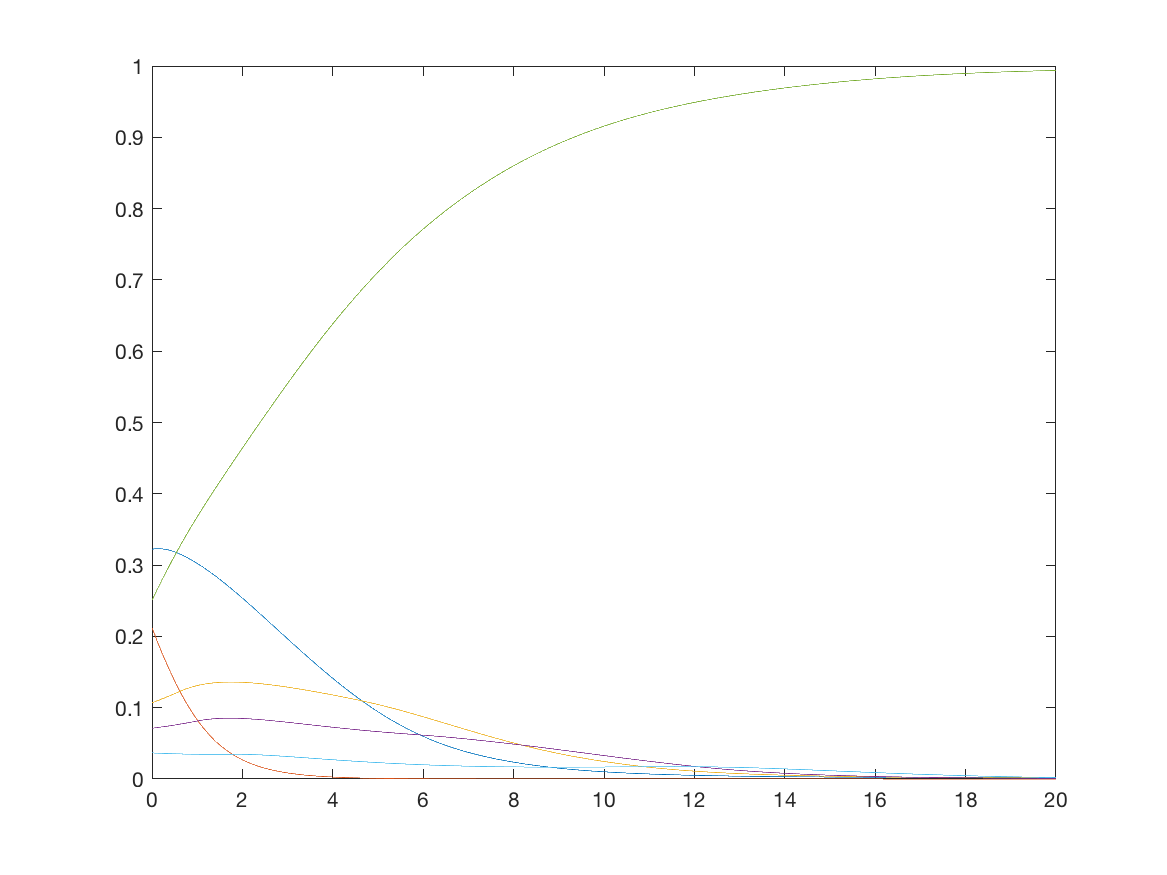}
\includegraphics[width=7.5cm]{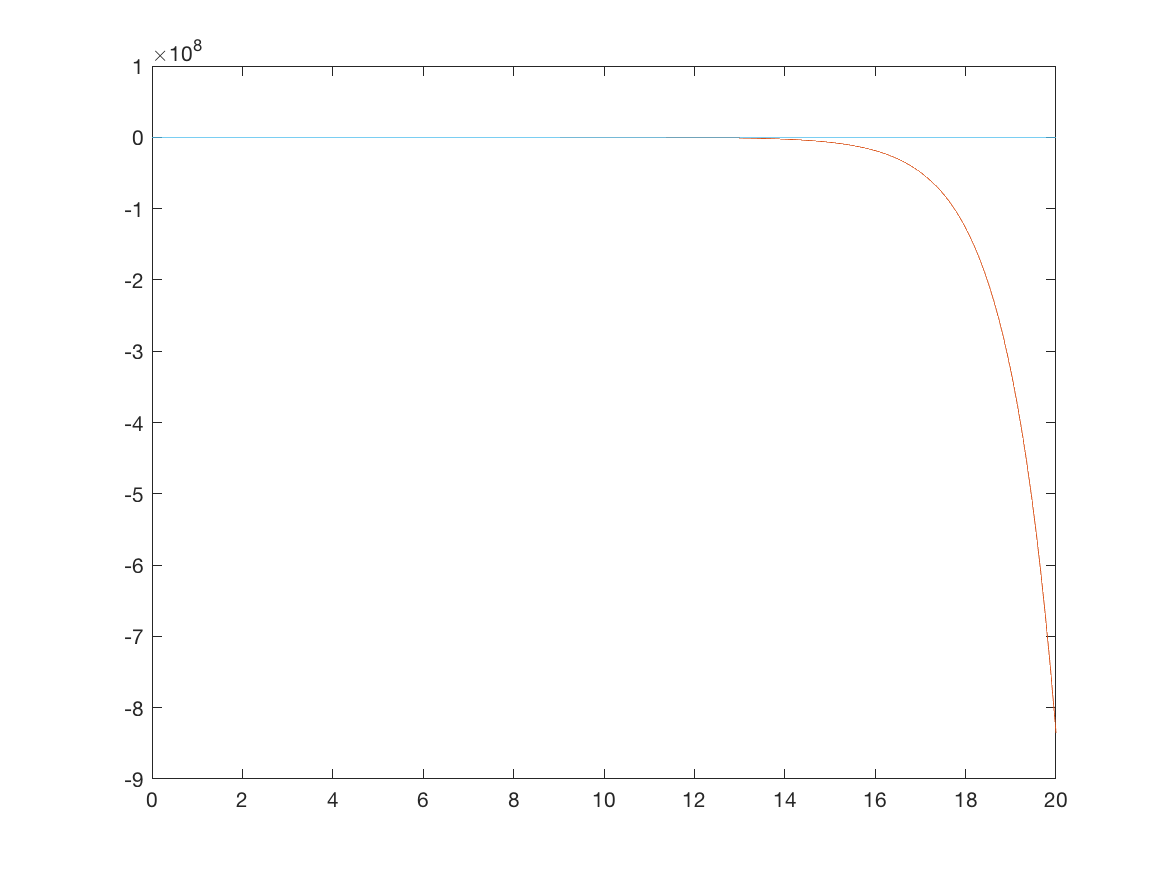}
\caption{Plot of $\rho_i(t)$(Left) and $S_i(t)$(Right) graphs with the initial data \eqref{IN1}. Since only one $\rho_i(t)$ converges to $1$ and others converge to zero, we can observe the complete synchronization on the graph.
}\label{Fig7}
\end{figure}

We also consider the other initial data:
\begin{align}\label{IN2}
\begin{cases}
\rho^0=(0.1524,~0.0910,~0.0698,~0.1583,~0.3424,~0.1862),\\
S^0=(-0.4890,~-0.4542,~-0.2708,~-0.6929,~1.0627,~0.1228).
\end{cases}
\end{align}
We present the numeric solution of system \eqref{F-4} with initial data \eqref{IN2} in Figure \ref{Fig8}.
\begin{figure}[h]
\includegraphics[width=7.5cm]{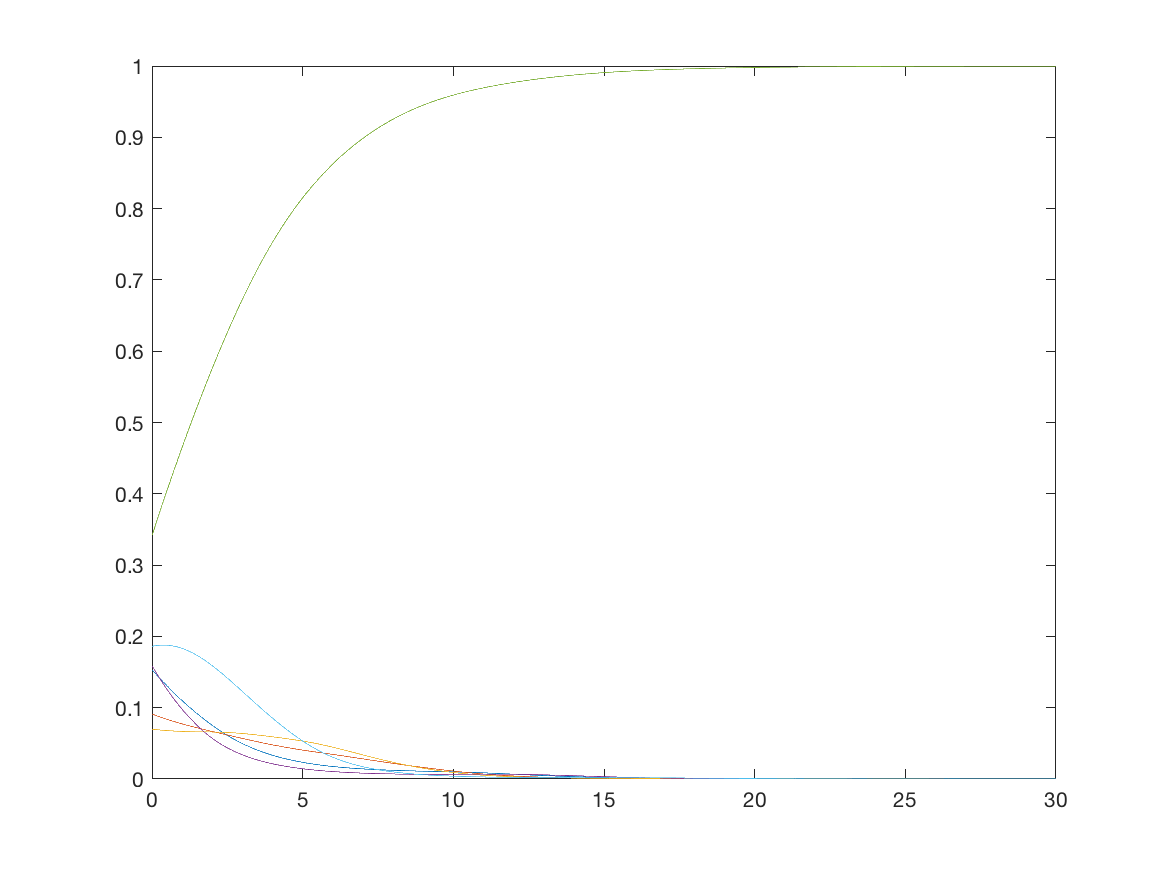}
\includegraphics[width=7.5cm]{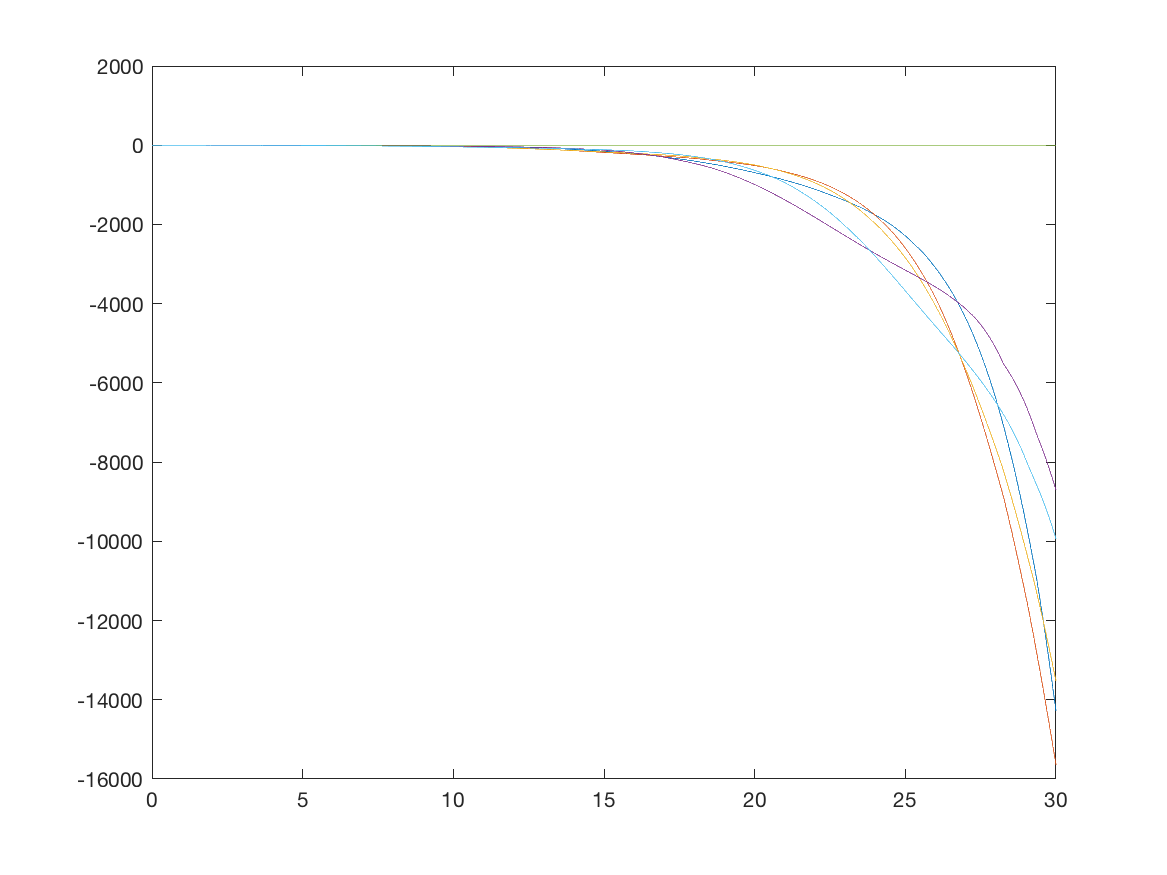}
\caption{Plot of $\rho_i(t)$(Left) and $S_i(t)$(Right) graphs with the initial data \eqref{IN2}. Since only one $\rho_i(t)$ converges to $1$ and others converge to zero, we can observe the complete synchronization on the graph.}\label{Fig8}
\end{figure}


\appendix
\section{Detail Proofs}
\setcounter{equation}{0}

\begin{lemma}\label{LA.1}
Suppose that the initial data $(x^0, p^0)$ satisfy the following condition:
\begin{align}\label{D-6-1}
p^0=-\nabla U(x^0),
\end{align}
and let $(x, p)$ be a solution to system \eqref{D-6}. Then, we have
\[
p(t)=-\nabla U(x(t))\quad\forall~t\geq0.
\]
\end{lemma}

\begin{proof}
We have the following calculation:
\begin{align*}
\frac{d}{dt}(p+\nabla U(x))_i
&=\frac{d}{dt}p_i+\partial_{ij}U(x)\frac{dx_j}{dt}\\
&=\frac{1}{2}\partial_i|\nabla U(x)|^2+\partial_{ij} U(x) p_j\\
&=\partial_{ij}U(x) \partial_jU(x)+\partial_{ij} U(x) p_j\\
&=\partial_{ij} U(x)(p+\nabla U(x))_j,
\end{align*}
where we use the Einstein's convention. Then, we get the following inequality:
\begin{align*}
\frac{1}{2}\frac{d}{dt}\|p+\nabla U(x)\|^2&=(p+\nabla U(x))_i\partial_{ij} U(x)(p+\nabla U(x))_j\\
&\leq \| \nabla^2 U(x)\|_F\cdot\|p+\nabla U(x)\|^2.
\end{align*}
From the initial condition $\|p(0)+\nabla U(x(0))\|=0$, we have
\[
\|p+\nabla U(x)\|^2=0
\]
for all $t\geq0$.
\end{proof}

\begin{lemma}\label{LA.2}
Let $\rho_t$ be a solution to \eqref{D-1}. If we set
\begin{align}\label{Z-1}
S_t=-\frac{\delta\mathcal{F}(\rho_t)}{\delta\rho_t},
\end{align}
then $(\rho_t, S_t)$ is a solution to \eqref{D-2} and the initial data satisfies \eqref{D-2-1}.
\end{lemma}

\begin{proof}
Let $\rho_t$ be a solution to \eqref{D-1}. From the substitution \eqref{Z-1} and \eqref{D-1}, we get
\[
\partial_t\rho_t+\mathrm{div}(\rho_t\nabla S_t)=0.
\]
So we have the first equation of \eqref{D-2}. By direct calculations, we also have
\begin{align*}
\partial_tS_t(x)&=-\int\frac{\delta^2\mathcal{F}(\rho_t)}{\delta\rho_t^2}(x, y)\partial_t\rho_t(y)dy
=-\int\frac{\delta^2\mathcal{F}(\rho_t)}{\delta\rho_t^2}(x, y)\mathrm{div}\left(\rho_t\nabla\left(\frac{\delta\mathcal{F}(\rho_t)}{\delta\rho_t}\right)\right)(y)dy\\
&=\int\rho_t(y)\frac{\delta}{\delta\rho_t(x)}\left|\nabla\left(\frac{\delta\mathcal{F}(\rho_t)}{\delta\rho_t}\right)(y) \right|^2dy\\
&= \frac{\delta}{\delta\rho_t(x)}\left(\int\rho_t\left|\nabla\left(\frac{\delta\mathcal{F}(\rho_t)}{\delta\rho_t}\right) \right|^2dy\right)-\left|\nabla\left(\frac{\delta\mathcal{F}(\rho_t)}{\delta\rho_t}\right) (x)\right|^2\\
&=\frac{\delta}{\delta\rho_t(x)}\left(\frac{1}{2}\int |\nabla S_t|^2\rho_t dy\right)-\frac{1}{2}|\nabla S_t(x)|^2.
\end{align*}
Finally, we have
\[
\partial_t S_t+\frac{1}{2}|\nabla S_t|^2=\frac{\delta}{\delta\rho_t}\left(\frac{1}{2}\int|\nabla S_t|^2\rho_tdx\right).
\]
This implies that
\[
\partial_t S_t+\frac{1}{2}|\nabla S_t|^2=\frac{\delta}{\delta\rho_t}\left(\frac{1}{2}\int\left|\nabla \left(\frac{\delta\mathcal{F}(\rho_t)}{\delta\rho_t}\right)\right|^2\rho_t(x)dx\right).
\]
Finally, the induced Wasserstein Hamiltonian flow from system \eqref{D-1} can be written as follows:
\begin{align*}
\begin{cases}
\partial_t\rho_t+\mathrm{div}(\rho_t\nabla S_t)=0,\\
\displaystyle\partial_t S_t+\frac{1}{2}|\nabla S_t|^2=\frac{\delta}{\delta\rho_t}\left(\frac{1}{2}\int\left|\nabla \left(\frac{\delta\mathcal{F}(\rho_t)}{\delta\rho_t}\right)\right|^2\rho_t(x)dx\right).
\end{cases}
\end{align*}
This is the desired result.
\end{proof}

\begin{lemma}\label{LA.3}
Suppose that the initial condition $(\xi_0, \xi^*_0)$ satisfies
\[
\xi_0(x)\equiv0 \quad\forall~ x\in M,
\]
and a pair of $C^2$ function $(\xi, \xi^*)$ is a solution to system \eqref{D-3-1}. Then, $\xi_t(x)$ is identically zero for all $t\geq0$ and $x\in M$.
\end{lemma}

\begin{proof}
First, we define the following functional:
\[
\mathcal{A}(t)=\int \xi_t(x)^2dx.
\]
From the previous result \eqref{D-3-1}, we have the following calculation:
\begin{align*}
\partial_t\mathcal{A}&=2\int \xi_t(x)\nabla\xi_t(x)\cdot\nabla\xi^*(x)dx-\iint \xi_t(x)[\delta^2\mathcal{F}](x, u)\nabla_u\cdot(\rho_t(u)\nabla_u\xi_t(u))dxdu\\
&=:\mathcal{I}_1+\mathcal{I}_2.
\end{align*}
Now, we estimate each terms for $0\leq t\leq T$. First, we estimate $\mathcal{I}_1$:
\begin{align*}
|\mathcal{I}_1|&=\left|\int \nabla \xi_t(x)^2\cdot\nabla\xi_t^*(x)dx\right|\\
&=\left|\int \xi_t(x)^2\Delta\xi_t^*(x)dx\right|\leq \left(\sup_{t\in[0, T], x\in M}|\Delta\xi_t^*(x)|\right)\mathcal{A}(t).
\end{align*}
Second, we estimate $\mathcal{I}_2$:
\begin{align*}
|\mathcal{I}_2|&=\left|\iint \xi_t(x)\xi_t(u) \nabla_u\cdot(\rho_t(u)\nabla_u[\delta^2\mathcal{F}](x, u))dxdu\right|\\
&\leq \left(\sup_{t\in[0, T], x, u\in M}| \nabla_u\cdot(\rho_t(u)\nabla_u[\delta^2\mathcal{F}](x, u))|\right) \mathcal{A}(t)\mathrm{Vol}(M).
\end{align*}
From the above estimates, for any $T>0$, we can find a fixed positive constant $C_T$ which satisfies the follows:
\[
\partial_t\mathcal{A}(t)\leq C_T\mathcal{A}(t)\quad\forall 0\leq t\leq T.
\]
This property yields
\[
\mathcal{A}(t)\leq \mathcal{A}(0)e^{C_Tt}\quad\forall~0\leq t\leq T.
\]
From the initial condition $\mathcal{A}(0)=0$, we get $\mathcal{A}(t)=0$ for all $0\leq t\leq T$. Since this argument holds for all $T>0$, we can prove that $\mathcal{A}(t)=0$ for all $t\geq0$. This yields $\xi_t(x)\equiv 0$ for all $x\in M$ and $t\geq0$.
\end{proof}

\end{document}